\documentclass[a4paper,11pt]{article}

\usepackage[latin1]{inputenc}
\usepackage{amsthm,amsfonts,amsmath}
\usepackage{graphicx,psfrag}
\usepackage{hyperref}
\usepackage{a4wide}
\usepackage{enumerate}

\newtheorem{proposition}{Proposition}[section]
\newtheorem{theorem}[proposition]{Theorem}
\newtheorem{lemma}[proposition]{Lemma}
\newtheorem{corollary}[proposition]{Corollary}
\newtheorem{definition}[proposition]{Definition}

\newenvironment{proofof}[1]{\smallskip\noindent{\textbf{Proof~of~#1.}}%
  \hspace{1pt}}{\hspace{-5pt}{\nobreak\quad\nobreak\hfill\nobreak%
    $\square$\vspace{2pt}\par}\smallskip\goodbreak}

\numberwithin{equation}{section}
\allowdisplaybreaks[4]

\renewcommand{\phi}{\varphi}
\renewcommand{\theta}{\vartheta}
\renewcommand{\epsilon}{\varepsilon}
\renewcommand{\L}[1]{\mathbf{L^#1}}

\newcommand{\C}[1]{\mathbf{C^{#1}}}

\newcommand{\Cb}[1]{\mathbf{C_b^{#1}}}

\newcommand{\BV}{\mathbf{BV}}

\newcommand{\modulo}[1]{{\left|#1\right|}}
\newcommand{\norma}[1]{{\left\|#1\right\|}}
\newcommand{\reali}{{\mathbb{R}}}
\newcommand{\naturali}{{\mathbb{N}}}
\newcommand{\Lip}{\mathop\mathbf{Lip}}
\newcommand{\tv}{\mathrm{TV}}
\renewcommand{\O}{\mathinner{\mathcal{O}(1)}}

\newcommand{\Id}{\mathop{\mathbf{Id}}}

\renewcommand{\d}[1]{\mathinner{\mathrm{d}{#1}}}

\begin{document}

\title{Optimization in Structure Population Models \\ through the
  Escalator Boxcar Train}

\author{Rinaldo M.~Colombo$^1$ \qquad Piotr Gwiazda$^2$ \qquad
  Magdalena Rosi\'nska$^3$}

\footnotetext[1]{University of Brescia,
  \texttt{rinaldo.colombo@unibs.it}}

\footnotetext[2]{University of Warsaw, \texttt{pgwiazda@mimuw.edu.pl}}

\footnotetext[3]{University of Warsaw, \texttt{mrosinska@pzh.gov.pl}}

\maketitle

\begin{abstract}
  \noindent The Escalator Boxcar Train (EBT) is a tool widely used in
  the study of balance laws motivated by structure population
  dynamics. This paper proves that the approximate solutions defined
  through the EBT converge to exact solutions. Moreover, this method
  is rigorously shown to be effective also in computing optimal
  controls. As preliminary results, the well posedness of classes of
  PDEs and of ODEs comprising various biological models is also
  obtained. A specific application to welfare policies illustrates the
  whole procedure.

  \medskip


  \medskip

\end{abstract}

\section{Introduction}
\label{sec:I}

This paper is devoted to the well posedness, to the numerical
approximation and to the optimal control of renewal equations
motivated by physiologically structured population models and whose
solutions attain values in spaces of measures.

The dynamics of populations which are heterogenous with respect to
some individual property can be described through initial -- boundary
value problems for a class of nonlinear first order partial
differential equations (PDE), called renewal equations. Within this
class, one of the first PDE models devoted to population biology is
the renewal equation introduced by Kermack and McKendrick with
reference to epidemiology, see~\cite{Kermack1991a,
  Kermack1991}. There, the time since infection, i.e., the age, plays
the role of a structure parameter, due to its essential role in the
spreading of the epidemic. Equations of the same class are later
proposed by von F\"orster in~\cite{Fo1959} to describe the process of
cell division. The recent monograph~\cite{DeRoosBook} provides an
extensive theoretical and empirical treatment of the ecology of
ontogenetic growth and development of organisms, emphasizing the
importance of an individual--based perspective in understanding the
dynamics of populations and communities.  Classical analytic studies
on these equations are settled in $\L1$ and go back, for instance, to
the monographs of Webb~\cite{Webb}, Iannelli~\cite{Ia1995} or
Thieme~\cite{Th2003}.

The space of positive Radon measures is introduced in biological
applications in~\cite{MeDi1986}. Indeed, whenever the distribution of
individuals is concentrated on discrete values of structure
parameters, for instance at the initial time, the resulting population
density may well lack absolute continuity with respect to the Lebesgue
measure. One is thus lead to consider the problem
\begin{equation}
  \label{eq:3}
  \left\{
    \begin{array}{l}
      \partial_t \mu
      +
      \partial_a \left(b (t,\mu) \, \mu\right)
      +
      c (t,\mu) \, \mu
      =
      0
      \\
      \displaystyle
      \left(b (t,\mu)\right) (0) \, D_\lambda\mu (0+)
      =
      \int_0^{+\infty} \beta (t,\mu) \, \d{\mu} 
      \\
      \mu (0) = \mu_o
    \end{array}
  \right.
\end{equation}
where $t \in \reali_+$ is time and $a \in \reali_+$ is a biological
parameter, typically age or size. The unknown $\mu$ is a time
dependent, non-negative and finite Radon measure. The growth function
$b$ and the mortality rate $c$ are strictly positive, while the birth
function $\beta$ is non-negative. By $D_\lambda\mu (0+)$ we denote the
Radon--Nikodym derivative of $\mu$ with respect to the Lebesgue
measure $\lambda$ computed at $0$. The initial datum $\mu_o$ is a
non-negative Radon measure.

The analysis of solutions to~\eqref{eq:3} in spaces of positive Radon
measures was initiated in~\cite{DiGe2005}, where the authors show the
weak$*$ continuity of solutions with respect to time and initial
data. They also point out the key relevance of the dependence of
solutions on the various model parameters, which was obtained
in~\cite{CarrilloColomboGwiazdaUlikowska, GwiazdaThomasEtAl,
  GwiazdaMarciniak, Ul2012}.

The Lipschitz continuous dependence of solutions in measure spaces
from time and initial datum is a preliminary step towards the
convergence of the so called \emph{particle methods}. These are
numerical algorithms whose starting idea is the representation of a
heterogeneous population as a sum of Dirac masses evolving in
time. This representation is consistent with the usual experimental
attitude of concentrating real data in discrete cohorts evolving in
time. On these grounds, the numerical algorithm usually referred to as
the \emph{Escalator Boxcar Train} (EBT) is introduced back
in~\cite{deRoos}. Remarkably, in spite of the wide success of this
method, a convergence proof of the EBT appears only rather recently
in~\cite{Linus}. A key role in this result is played by the bounded
Lipschitz distance introduced in~\cite{GwiazdaThomasEtAl}. Detailed
estimates on the order of convergence are then provided
in~\cite{GwiazdaEtAl}.

Another numerical method effective in the computation of solutions to
structured population models is proposed in~\cite{CaGwUl2013}. Here, a
key role is played by the \emph{operator splitting} method. According
to it, the measure valued semigroup generated by renewal equations can
be approximated through the iterated application of simpler
semigroups. More precisely, a problem involving both transport terms
and nonlocal growth terms is approximated through two problems, each
involving only one of the two processes. The analytic framework
estabilished in~\cite{CarrilloColomboGwiazdaUlikowska} allows a
detailed control of the convergence rate of the algorithm.

From the measure theoretic point of view, the above mentioned results
rely on the use of Wasserstein (or Monge-Kantorovich) type metrics,
adapted to the nonconservative character of~\eqref{eq:3}. This
methodology was proposed in~\cite{GwiazdaThomasEtAl} for a flat metric
(bounded Lipschitz distance) and in~\cite{GwiazdaMarciniak} for a
Wasserstein metric, suitably modified to deal with nonnegative Radon
measures with integrable first moment. A relevant advantage of this
approach is in providing a structure of a space appropriate both to
compare solutions and to study their stability. Remark that precise
estimates on the continuous dependence of solutions on the modeling
parameters plays a key role in the numerical approximations and in
calibrating the model calibration on the basis of experimental
data. We refer to~\cite{piccoli2014} for the definition and
propertiesof a similar metric structure.

Similar techniques based on particle methods are usual tools in
simulating kinetic models since more than three decades in physics,
see for instance~\cite{Ra1985} and the references therein. Recent
applications include for instance the porous medium
equation~\cite{Ot2001, WeWi2010} and the isentropic Euler equations in
fluid mechanics~\cite{GaWe2009, WeWi2010}. Other artcle method are
found also in the study of problems related to crowd dynamics and
pedestrians flow, see~\cite{evers2014, evers2015, piccoli2014}, as
well as in the description of the collective motion of large groups of
agents, see~\cite{CaCaRo2011}. Differently from the case of structured
population models, the original particle methods are mainly designed
for problems where the total mass, or number of individuals, is
conserved.

\smallskip

Aiming at the optimal control of the solution to~\eqref{eq:3}, we
introduce therein a control parameter $u$, possibly time and/or state
dependent, attaining values in a given set $\mathcal{U}$. Therefore,
we obtain:
\begin{equation}
  \label{eq:5}
  \left\{
    \begin{array}{l}
      \partial_t \mu
      +
      \partial_a \left(b (t,\mu;u) \, \mu\right)
      +
      c (t,\mu;u) \, \mu
      =
      0
      \\
      \displaystyle
      \left(b (t,\mu;u)\right) (0) \, D_\lambda\mu (0+)
      =
      \int_0^{+\infty} \beta (t,\mu;u) \, \d{\mu}
      \\
      \mu (0) = \mu_o
    \end{array}
  \right.
\end{equation}
Together with~\eqref{eq:5}, we are given a cost functional
\begin{displaymath}
  \mathcal{J} (u)
  =
  \int_0^{+\infty} j\left(t,u (t),\mu(t)\right) \d{t}
\end{displaymath}
and we provide below a \emph{constructive} algorithm to find, within a
suitable function space, a control function $u_*$ optimal in the sense
that
\begin{displaymath}
  \mathcal{J} (u_*) = \min_{u (t) \in \mathcal{U}} \mathcal{J} (u) \,.
\end{displaymath}
As is well known, solutions to conservation or balance laws typically
depend in a Lipschitz continuous way from the initial datum as well as
from the functions defining the equation. This does not allow the use
of differential tools in the search for the optimal control.

Here, \emph{constructive} should be understood in the following sense:
on the basis of the control problem for~\eqref{eq:5}, we define a
sequence of control problems for a system of ordinary differential
equations and prove that the corresponding sequence of optimal
controls converges to an optimal control for the original
problem. More precisely, we approximate the solution to~\eqref{eq:5}
by means of the EBT algorithm as defined
in~\cite[Section~III]{GwiazdaEtAl}. The functional $\mathcal{J}$
computed along approximate solutions is proved to be a smooth, namely
$\C1$, function of the control parameter $u$ and this allows to
exhibit the existence of an optimal control for each approximate
problem. A limiting procedure constructively ensures the existence of
the optimal control for the original problem~\eqref{eq:5}.

\smallskip

The next section presents results on the well posedness
of~\eqref{eq:3} and the results on the escalator boxcar train
algorithm that allow to obtain our main result, namely the
construction of a sequence of controls that converge to an optimal
control for~\eqref{eq:5}. Section~\ref{sec:E} is devoted to a possible
application of the theory here developed. The technical proofs are
deferred to Section~\ref{sec:T}, with a final Appendix that gathers
necessary results concerning ordinary differential equations.

\section{Main Results}
\label{sec:A}

Throughout, we denote $\reali_+ = \left[0, +\infty\right[$. Let
$(M,d_M)$ be a metric space and $(V, \norma{\ }_V)$ be a normed
space. Then, $\C0 (M; V)$, respectively $\C{0,1} (M; V)$ is the space
of continuous, respectively Lipschitz continuous, functions defined on
$M$ and attaining values in $V$, equipped with the norm
\begin{eqnarray}
  \label{eq:dC0}
  \norma{\phi}_{\C0 (M,V)}
  & = &
  \sup_{x \in M} \norma{\phi (x)}_{V}
  \,,\mbox{ respectively}
  \\
  \label{eq:dLip}
  \norma{\phi}_{\C{0,1} (M,V)}
  & = &
  \max\left\{
    \sup_{x \in M} \norma{\phi (x)}_{V}, \,
    \sup_{x_1,x_2\in M}\frac{\norma{\phi (x_2) - \phi (x_1)}_V}{d_M (x_1,x_2)}
  \right\} \,.
\end{eqnarray}
\noindent Given $T \in \reali$ and a function $f\colon [0,T] \to V$,
we set
\begin{equation}
  \label{eq:TV}
  \tv_{V} (f)
  =
  \sup\left\{
    \sum_{i=1}^n
    \norma{f (t_i)-f (t_{i-1})}_{V}
    \colon
    n \in \naturali
    \mbox{ and }
    \begin{array}{@{}c@{\,}r}
      t_i \in [0,T] & \mbox{ for } i=0, \ldots, n
      \\
      t_{i-1}<t_i & \mbox{ for } i=1, \ldots, n
    \end{array}
  \right\} \,.
\end{equation}
The space $\mathcal{M}^+ (\reali_+)$ of positive Radon measure on
$\reali_+$ is equipped with the flat distance
\begin{equation}
  \label{eq:d}
  d (\mu', \mu'')
  =
  \sup
  \left\{
    \int_{\reali_+}\phi \, \d{(\mu'-\mu'')} \colon
    \phi \in \C1 (\reali_+; [-1,1]) \mbox{ with } \Lip (\phi) \leq 1
  \right\} \,,
\end{equation}
see~\cite[Section~2]{CarrilloColomboGwiazdaUlikowska}.  Below, for
positive $T, \mathcal{L}$ and $\mathcal{C}$, we use the space
$\mathcal{F}$ of functions
\begin{displaymath}
  f \colon [0,T] \to
  \C{0,1} ( \mathcal{M}^+ (\reali_+) \times \reali_+;\reali_+)
\end{displaymath}
with the properties:
\begin{enumerate}[$(\mathcal{F}_1)$]
\item $f$ is bounded: $\norma{f (t)}_{\C{0,1} (\mathcal{M}^+
    (\reali_+) \times \reali_+;\reali)} \leq \mathcal{L}$ for all $t
  \in [0,T]$.
\item $f$ has bounded total variation in time: $\tv_{\C0
    (\mathcal{M}^+ (\reali_+) \times \reali_+; \reali_+)} (f ) \leq
  \mathcal{C}$.
\end{enumerate}
Throughout, the constants $T, \mathcal{L}$ and $\mathcal{C}$ are kept
fixed and the dependence of $\mathcal{F}$ on them is omitted. In
$\mathcal{M}^+ (\reali_+) \times \reali_+$ we use the distance
\begin{displaymath}
  d_{\mathcal{M}^+ (\reali_+) \times \reali_+}
  \left((\mu_1,a_1), (\mu_2,a_2)\right)
  =
  d (\mu_1,\mu_2) + \modulo{a_2 - a_1} \,,
\end{displaymath}
where $d$ is as in~\eqref{eq:d}. Therefore, $(\mathcal{F}_1)$ also
implies that $f$ is Lipschitz continuous in $\mu$ and $a$ uniformly in
$t$, in the sense that for all $t \in [0,T]$, $\mu_1,\mu_2 \in
\mathcal{M}^+ (\reali_+)$ and $a_1,a_2 \in \reali_+$,
\begin{displaymath}
  \modulo{\left(f (t)\right)(\mu_1,a_1) - \left(f (t)\right)(\mu_2,a_2)}
  \leq
  \mathcal{L} \left(d (\mu_1,\mu_2) + \modulo{a_1 - a_2}\right) \,.
\end{displaymath}

\subsection{PDE -- Well Posedness}
\label{sub:pde}

As a first step, we need to extend the well posedness of~\eqref{eq:3}
obtained in~\cite[Theorem~2.11]{CarrilloColomboGwiazdaUlikowska} to
the case of functions $b$ and $c$ being only of bounded variation in
time. First, recall the definition of solution to~\eqref{eq:3}
attaining as values Radon measures.

\begin{definition}{~\cite[Definition~3.1]{GwiazdaLorenz}}
  \label{def:pdeSol}
  Fix $T>0$ and let $\mu_o \in \mathcal{M}^+ (\reali_+)$. By
  \emph{solution} to~\eqref{eq:3} we mean a function $\mu \colon [0,T]
  \to \mathcal{M}^+ (\reali_+)$ with the following properties:
  \begin{enumerate}
  \item $\mu$ is Lipschitz continuous with respect to the flat
    distance~\eqref{eq:d};
  \item for all $\phi \in (\C1 \cap \C{0,1}) ([0,T] \times \reali_+;
    \reali)$
    \begin{eqnarray*}
      & &
      \int_{\reali_+} \phi (T,a) \d{\mu_t (a)}
      -
      \int_{\reali_+} \phi (0,a) \d{\mu_o (a)}
      \\
      & = &
      \int_0^T \int_{\reali_+} \partial_t \phi (t,a) \d{\mu_t (a)} \d{t}
      \\
      & &
      +
      \int_0^T \int_{\reali_+}
      \left(
        \partial_a \phi (t,a) \; \left(b (t,\mu_t)\right) (a)
        +
        \phi (t,a) \; \left(c (t,\mu_t)\right) (a)
      \right)
      \d{\mu_t (a)} \d{t}
      \\
      & &
      +\int_0^T \int_{\reali_+}
      \phi (t,0) \; \left(\beta (t,\mu_t)\right)\! (a) \,
      \d{\mu_t (a)} \d{t} \,.
    \end{eqnarray*}
  \end{enumerate}
\end{definition}

We now weaken the assumptions on the regularity in time used
in~\cite[Theorem~2.11]{CarrilloColomboGwiazdaUlikowska}.

\begin{theorem}
  \label{thm:pdeWP}
  Fix $T>0$. Let $b,c,\beta \in \mathcal{F}$. Then, for any $\mu_o \in
  \mathcal{M}^+ (\reali_+)$, problem~\eqref{eq:3} admits a unique
  solution in the sense of Definition~\ref{def:pdeSol}. Moreover,
  there exists a constant $C$ dependent only on $\mathcal{C},
  \mathcal{L}$ and $T$ such that if for $i=1,2$, $\mu^i$ is the
  solution to~\eqref{eq:3} with initial data $\mu_o^i$ and $b,c,\beta$
  replaced by $b_i, c_i, \beta_i$, then,
  \begin{equation}
    \label{eq:pdeStab}
    \begin{array}{@{}rcl@{}}
      d\left(\mu^1 (t) , \mu^{2} (t)\right)
      & \leq &
      \displaystyle
      d (\mu_o^1,\mu_o^2) \, e^{C\, t}
      \\[4pt]
      & &
      \displaystyle
      +
      C \, t \, e^{C\, t}
      \Bigl( \,
      \sup_{t \in[0,T]}
      \norma{b_1 (t) -b_{2} (t)}_{\C0 (\mathcal{M}^+ (\reali_+) \times \reali_+;\reali)}
      \\[12pt]
      & &
      \displaystyle
      \qquad\qquad\qquad
      +
      \sup_{t \in[0,T]}
      \norma{c_1 (t) -c_{2} (t)}_{\C0 (\mathcal{M}^+ (\reali_+) \times \reali_+;\reali)}
      \\[12pt]
      & &
      \displaystyle
      \qquad\qquad\qquad
      +
      \sup_{t \in[0,T]}
      \norma{\beta_1 (t) -\beta_{2} (t)}_{\C0 (\mathcal{M}^+ (\reali_+) \times \reali_+;\reali)}
      \Bigr) \,.
    \end{array}
  \end{equation}
\end{theorem}

\noindent The proof is deferred to Section~\ref{sec:T}.

Aiming at the study of~\eqref{eq:5}, we extend the definition of
$\mathcal{F}$ as follows. Fix $T>0$ and a compact subset $\mathcal{U}$
of $\reali^N$, for fixed positive $T, \mathcal{L}, \mathcal{C}$ and a
positive integer $N$, we introduce the space $\mathcal{F}^u$ of
functions
\begin{displaymath}
  f \colon [0,T] \times \mathcal{U} \to \C{0,1} (\mathcal{M}^+ (\reali_+) \times \reali_+;\reali)
\end{displaymath}
with the properties:
\begin{enumerate}[$(\mathcal{F}^u_1)$]
\item $f$ is bounded: $\norma{f (t,u)}_{\C{0,1} (\mathcal{M}^+
    (\reali_+) \times \reali_+;\reali)} \leq \mathcal{L}$ for all $t
  \in [0,T]$ and all $u \in \mathcal{U}$.
\item $f$ has bounded total variation in $t$ uniformly in $u$:
  $\tv_{\C0(\mathcal{M}^+ (\reali_+) \times \reali_+; \reali_+)}
  \left(f (\cdot,u)\right) \leq \mathcal{C}$ for all $u \in
  \mathcal{U}$, for all $u \in \mathcal{U}$.
\item $f$ is Lipschitz continuous in the control uniformly in time:
  for all $t \in [0,T]$ and for all $u_1, u_2 \in \mathcal{U}$,
  $\norma{f (t,u_1) - f (t,u_2)}_{\C0(\mathcal{M}^+ (\reali_+) \times
    \reali_+;\reali)} \leq \mathcal{L} \, \norma{u_1-u_2}$.
\end{enumerate}
\noindent As above, we remark that $(\mathcal{F}^u_1)$ ensures that $f
(t;u)$ is Lipschitz continuous in $\mu$ and $a$ uniformly in $t$ and
$u$: for all $t \in [0,T]$, $u \in \mathcal{U}$, $\mu_1,\mu_2 \in
\mathcal{M}^+ (\reali_+)$ and $a_1,a_2 \in \reali_+$,
\begin{displaymath}
  \modulo{\left(f (t,u)\right) (\mu_1,a_1)
    -
    \left(f (t,u)\right) (\mu_2,a_2)}
  \leq
  \mathcal{L} \left(d (\mu_1,\mu_2) + \modulo{a_1 - a_2}\right) \,.
\end{displaymath}
\noindent In~$(\mathcal{F}^u_2)$, the total variation is computed as
in~\eqref{eq:TV}, keeping $u$ fixed.  Throughout, the constants $T,
\mathcal{L}$ and $\mathcal{C}$ are kept fixed and the dependence of
$\mathcal{F}^u$ on them is omitted.

The extension of Definition~\ref{def:pdeSol} from the case
of~\eqref{eq:3} to that of~\eqref{eq:5} is immediate.

\begin{corollary}
  \label{cor:pdeWP}
  Fix $T>0$ and a compact subset $\mathcal{U}$ of $\reali^N$. Let
  $b,c,\beta \in \mathcal{F}^u$. Then, for any $\mu_o \in
  \mathcal{M}^+ (\reali_+)$ and any $u \in \BV ([0,T];\mathcal{U})$,
  problem~\eqref{eq:5} admits a unique solution. Moreover, there
  exists a constant a constant $C$ dependent only on $\mathcal{L},
  \mathcal{C}$ and $T$ such that if for $i=1,2$, $\mu^i$ is the
  solution to~\eqref{eq:3} with initial data $\mu_o^i$, and
  $b,c,\beta,u$ replaced by $b_i,c_i,\beta_i, u_i$, then,
  \begin{equation}
    \label{eq:pdeStabControl}
    \begin{array}{@{}r@{\,}c@{\,}l@{}}
      d\left(\mu^1 (t) , \mu^{2} (t)\right)
      \leq
      \displaystyle
      d (\mu_o^1,\mu_o^2) \, e^{C\, t}
      +
      C \, t \, e^{C\, t}
      & \Bigl( &
      \displaystyle
      \sup_{t \in [0,T] }
      \norma{b_1 (t) -b_{2} (t)}_{
        \C0 (\mathcal{M}^+ (\reali_+) \times \mathcal{U}\times\reali_+;\reali)}
      \\[12pt]
      & &
      \displaystyle
      +
      \sup_{t \in [0,T] }
      \norma{c_1 (t) -c_{2} (t)}_{\C0 (\mathcal{M}^+ (\reali_+) \times \mathcal{U}\times\reali_+;\reali)}
      \\[12pt]
      & &
      \displaystyle
      +
      \sup_{t \in [0,T] }
      \norma{\beta_1 (t) -\beta_{2} (t)}_{\C0 (\mathcal{M}^+ (\reali_+) \times \mathcal{U}\times\reali_+;\reali)}
      \\[12pt]
      & &
      \displaystyle
      +
      \sup_{t \in [0,T] }\norma{u_1 (t) -u_{2} (t)}
      \Bigr) \,.
    \end{array}
  \end{equation}
\end{corollary}

\noindent The proof is in Section~\ref{sec:T}.

\subsection{ODE -- Well Posedness}
\label{sub:ode}

We first present the approximation algorithm introduced
in~\cite{deRoos}, see~\cite{Linus, GwiazdaEtAl} for the present
simplified version. Fix a positive time $T$. For any $n \in \naturali
\setminus \{0\}$ and for the time step $\Delta t$, approximate the
initial datum $\mu_o$ in~\eqref{eq:5} by means of a linear combination
$\mu_0^n$ of Dirac deltas centered at $x_0^0, x_0^1, \ldots, x_0^n$
with masses $m_0^1, \ldots, m_0^n$ and approximate the initial datum
with the measure
\begin{displaymath}
  \mu_0^n = \sum_{i=1}^n m_0^i \; \delta_{x_0^i} \,.
\end{displaymath}
On the time interval $\left[0, \Delta t\right[$, we approximate the
solution to~\eqref{eq:5} with the measure
\begin{displaymath}
  \mu^n (t)
  =
  \sum_{i=1}^n m^i (t) \; \delta_{x^i (t)}
\end{displaymath}
where
\begin{equation}
  \label{eq:ode0}
  \left\{
    \begin{array}{rcl@{}l}
      \dot x^i & = &
      \left(b\left(t, \mu^n (t); u (t) \right)\right)(x^i)
      & i = 0, \ldots, n
      \\
      \dot m^0 & = &
      \displaystyle
      -
      c\left(t, \mu^n (t); u (t) \right) (x^0) \; m^0
      +
      \sum_{i=1}^n & \, \beta\left(t, \mu^n (t); u (t) \right) (x^i) \; m^i
      \\
      \dot m^i & = &
      -
      c\left(t, \mu^n (t); u (t) \right) (x^i) \; m^i
      & i = 1, \ldots, n
      \\
      x^i (0) & = & x^i_0
      & i = 0, \ldots, n
      \\
      m^0 (0) & = & 0
      \\
      m^i (0) & = & m^i_0
      & i = 1, \ldots, n
    \end{array}
  \right.
\end{equation}
Define $x^i_1 = \lim_{t \to \Delta t-} x^i (t)$ and $m^i_1 = \lim_{t
  \to \Delta t-} m^i (t)$ for $i=0, \ldots, n$.  Iteratively, for $k
\geq 1$, we prolong $\mu^n$, $x^{-k+1}, \ldots, x^n$ and $m^{-k+1},
\ldots, m^n$ on the interval $\left[k \, \Delta t, (k+1) \, \Delta
  t\right[$ solving
\begin{equation}
  \label{eq:odek}
  \left\{
    \begin{array}{@{}rcl@{}l@{}}
      \dot x^i & = &
      \left(b\left(t, \mu^n (t); u (t) \right)\right)(x^i)
      & i = -k, \ldots, n
      \\
      \dot m^{-k} & = &
      \displaystyle
      -
      c\left(t, \mu^n (t); u (t) \right) (x^{-k}) \; m^{-k}
      +
      \sum_{i=-k+1}^n & \! \beta\left(t, \mu^n (t); u (t) \right) (x^i) \; m^i
      \\
      \dot m^i & = &
      -
      c\left(t, \mu^n (t); u (t) \right) (x^i) \; m^i
      & i = -k+1, \ldots, n
      \\
      x^{i} (k \Delta t) & = & x^i_k & i = -k+1, \ldots, n
      \\
      m^{i} (k \Delta t) & = & m^i_k & i = -k+1, \ldots, n
      \\
      x^{-k} (k \Delta t) & = & 0
      \\
      m^{-k} (k \Delta t) & = & 0
    \end{array}
  \right.
\end{equation}
where $x^i_k = \lim_{t \to k\Delta t-} x^i (t)$, $m^i_k = \lim_{t \to
  k \Delta t-} m^i (t)$ for $i=0, \ldots, n$ and
\begin{displaymath}
  \mu^n (t)
  =
  \sum_{i=-k+1}^n m^i (t) \; \delta_{x^i (t)} \,.
\end{displaymath}

To describe the hypotheses on $b,c,\beta$ ensuring the well posedness
of~\eqref{eq:ode0}--\eqref{eq:odek} it is of use to introduce, for
positive $T$ and $L$, the set $\widetilde{\mathcal{F}}^u$ of functions
\begin{displaymath}
  f \colon
  [0,T] \times \mathcal{U}
  \to
  \C{0,1}(\mathcal{M}^+(\reali_+) \times \reali_+;\reali_+)
\end{displaymath}
such that
\begin{displaymath}
  \left(f (t; u)\right) (\mu, a)
  =
  \tilde f \left(
    t,
    \int_{\reali_+} \bar f (\alpha) \d{\mu(\alpha)}, a; u
  \right)
\end{displaymath}
where
\begin{enumerate}[$({\widetilde{\mathcal{F}}}_1^u)$]
\item The map $\bar f \in \C1(\reali_+; \reali_+)$ is bounded.
\item The map $(A,a;u) \to \tilde f (t,A,a;u)$ is in $\C1 (\reali_+
  \times \reali_+ \times \mathcal{U}; \reali_+)$ for a.e.~$t \in
  [0,T]$.
\item The map $t \to \tilde f (t, A,a;u)$ is in $\L\infty ([0,T];
  \reali_+)$ for all $A \in \reali_+$, $a \in \reali_+$ and $u \in
  \mathcal{U}$.
\item $\tilde f$ is Lipschitz continuous in $A,a,u$ uniformly in $t$:
  \begin{displaymath}
    \modulo{\tilde f (t, A_1,a_1;u_1) - \tilde f (t, A_2,a_2;u_2)}
    \leq
    L \left(
      \modulo{A_1 - A_2}
      +
      \modulo{a_1 - a_2}
      +
      \norma{u_1 - u_2}
    \right) \,.
  \end{displaymath}
\end{enumerate}

\noindent The next result ensures the well posedness of the Cauchy
Problem for the system of ordinary differential
equations~\eqref{eq:fi}--\eqref{eq:gi}.

\begin{theorem}
  \label{thm:WP}
  Fix $n, N \in \naturali \setminus \{0\}$, $T, L > 0$ and a compact
  subset $\mathcal{U}$ of $\reali^N$.  Let $b, c, \beta \in
  \tilde{\mathcal{F}}^u$. Then, for any control $u \in \BV ([0,T];
  \mathcal{U})$ and any initial datum $(x_0^1, \ldots, x_0^n) \in
  \reali_+^{n+1}$, $(m_0^1, \ldots, m_0^n) \in \reali^n$,
  problem~\eqref{eq:ode0}--\eqref{eq:odek} admits a unique solution $t
  \to (x,m)_u (t)$ defined for all $t \in [0,T]$. Moreover, the map $u
  \to (x,m)_u$ is in $\C1\bigl(\BV ([0,T]; \mathcal{U}); \C0 ([0,T];
  \reali_+^{n+1} \times \reali^n)\bigr)$.
\end{theorem}

\noindent The proof directly follows from Lemma~\ref{lem:reg1} in
\S~\ref{sub:ode} and from the usual properties of the Nemitsky
operator.

\begin{theorem}
  \label{thm:Conv}
  Let $n \in \naturali \setminus \{0\}$, fix $T>0$ and a compact
  subset $\mathcal{U}$ of $\reali^N$.  Let $b, c, \beta \in
  \mathcal{F}^u$ and $u \in \BV ([0,T]; \mathcal{U})$.  Fix $\mu_o \in
  \mathcal {M}^+(\reali_+)$, $(x_o^1, \ldots, x_o^n) \in \reali_+^n$ ,
  $(m_o^1, \ldots, m_o^n) \in \reali_+^n$.  Let $\mu$ solve
  problem~\eqref{eq:5} in the sense of Definition~\ref{def:pdeSol} and
  $(m, x)$ solve problem~\eqref{eq:ode0}--\eqref{eq:odek} with time
  step $\Delta t$. Then, there exists a positive $C$ independent from
  $u$, $\Delta t$ and $n$ such that for all $t \in [0,T]$,
  \begin{displaymath}
    d\left(
      \mu(t),   \sum_{i=-n}^n m^i(t) \, \delta_{x^i(t)}
    \right)
    \leq
    C
    \cdot
    \left[
      \Delta t
      +
      d\left(\mu_o, \sum_{i=0}^n m_0^i \, \delta_{x_0^i}\right)
    \right] \,.
  \end{displaymath}
\end{theorem}

\noindent In specific numerical implementations of the present method,
the quantity $d \! \left(\mu_o, \sum_{i=0}^n m_0^i \,
  \delta_{x_0^i}\right)$ is typically of the same order of the size of
the space mesh $\Delta x$.

\subsection{Optimal Control}
\label{sub:control}

A general cost functional defined on the controls in $\BV([0,T];
\mathcal{U})$ is
\begin{equation}
  \label{eq:8}
  \begin{array}{ccccl}
    \widetilde{\mathcal{J}} & \colon & \BV([0,T]; \mathcal{U}) & \to & \reali
    \\
    & & u & \to &
    \displaystyle
    \int_0^{T} j \! \left(t,u(t),\int_{\reali_+}\gamma(\xi)
      \d{\mu_u(t)(\xi)}\right) \d{t}
  \end{array}
\end{equation}
where $\gamma \in \C{0,1} (\reali_+;\reali_+)$, $\mu_u$ is the
solution to~\eqref{eq:5} corresponding to the control $u$ with $b,
\beta, c$ and $\mu_o$ satisfying the assumptions of
Theorem~\ref{thm:pdeWP}, and $j \colon
[0,T]\times\mathcal{U}\times\reali_+\to \reali_+$ being such that:
\begin{enumerate}[$\bf(J_1)$]
\item $j \geq 0$
\item the map $t \to j(t,x,u)$ is measurable for all $x \in \reali_+$,
  $u \in \mathcal{U}$ and there exists a $\hat u \in \BV ([0,T];
  \mathcal{U})$ such that $\mathcal{J} (\hat u) < +\infty$
\item there exist $L \in \L1([0,T]; \reali_+)$ and a nondecreasing
  $\omega \in \C0(\reali_+; \reali_+)$, with $\omega(0) = 0$, such
  that
  \begin{displaymath}
    \modulo{j(t,x_1,u_1)-j(t,x_2,u_2)}
    \leq
    L(t) \;\; \omega \! \left(\modulo{x_1-x_2}+\modulo{u_1-u_2}\right)
  \end{displaymath}
  for a.e.~$t \in [0,T]$, for all $x_1,x_2 \in \reali_+$ and all $u_1,
  u_2 \in \mathcal{U}$.
\end{enumerate}
\noindent Having to consider also costs related to the adjustments in
the values of the control, it is natural to seek the minimization of
\begin{equation}
  \label{eq:9}
  \begin{array}{ccccl}
    \mathcal{J} & \colon  & \BV([0,T]; \mathcal{U}) & \to & \reali
    \\
    & & u & \to &
    \displaystyle
    \widetilde{\mathcal{J}} (u) + \tv_{\reali^N} (u) \,.
  \end{array}
\end{equation}

As a first result, we prove the existence of an optimal control.

\begin{theorem}
  \label{thm:optPDE}
  Fix $T>0$ and a compact subset $\mathcal{U}$ of $\reali^N$. For all
  $b, c, \beta \in \mathcal{F}^u$, $u \in \BV ([0,T];\mathcal{U})$ and
  $\mu_o \in \mathcal{M}^+ (\reali_+)$, let $\mu_u$ be the solution to
  problem~\eqref{eq:5}. With reference to the cost
  functional~\eqref{eq:8}, $\gamma\in \C{0,1}(\reali_+;\reali_+)$ and
  $j$ satisfies~$\bf(J_1)$, $\bf(J_2)$, $\bf(J_3)$. Then, there exists
  a control minimizing $\mathcal{J}$ as defined in~\eqref{eq:9}:
  \begin{displaymath}
    \exists \, u^* \in \BV ([0,T];\mathcal{U}) \; \colon \quad
    \mathcal{J}(u^*) = \inf_{u\in \BV ([0,T];\mathcal{U})} \mathcal{J}(u) \,.
  \end{displaymath}
\end{theorem}

We now pass to the discrete counterpart of Theorem~\ref{thm:optPDE},
substituting the evolution described by~\eqref{eq:5} with the
approximation provided by the Escalator Boxcar
Train~\eqref{eq:ode0}--\eqref{eq:odek}. At the same time, also the
functionals~\eqref{eq:ode0}--\eqref{eq:odek} have to be computed on
linear combination of Dirac deltas.

\begin{theorem}
  \label{thm:optODE}
  Fix $T>0$ and a compact subset $\mathcal{U}$ of $\reali^N$.  Let
  $b,c,\beta \in \mathcal{F}^u$ and $u \in \BV ([0,T];
  \mathcal{U})$. For any $n \in \naturali \setminus \{0\}$ and
  ${\Delta t}_n > 0$, fix an initial datum $(x_o^1, \ldots, x_o^n) \in
  \reali_+^{n+1}, (m_o^1, \ldots, m_o^n) \in \reali^n$
  in~\eqref{eq:ode0}--\eqref{eq:odek} and call $(x^{-n}, \ldots,
  x^n)$, $(m^{-n}, \ldots, m^n)$ the corresponding solution. Further,
  define the cost functionals
  \begin{eqnarray}
    \label{eq:10}
    & &
    \begin{array}{ccccl}
      \widetilde{\mathcal{J}}_n
      & \colon & \BV([0,T]; \mathcal{U}) & \to & \reali
      \\
      & & u & \to &
      \displaystyle
      \int_0^{T} j \! \left(t,u(t),\int_{\reali_+}\gamma(\xi)
        \d{\mu_u^n(t)(\xi)}\right) \d{t}
    \end{array}
    \\[4pt]
    \label{eq:11}
    & &
    \begin{array}{ccccl}
      \mathcal{J}_n & \colon  & \BV([0,T]; \mathcal{U}) & \to & \reali
      \\
      & & u & \to &
      \displaystyle
      \widetilde{\mathcal{J}}_n (u) + \tv_{\reali^N} (u) \,.
    \end{array}
  \end{eqnarray}
  where $\mu_u^n(t) = \sum_{i=-n}^n m^i(t) \, \delta_{x^i(t)}$,
  $\gamma\in (\C1 \cap \C{0,1})(\reali_+;\reali_+)$, $j$
  satisfies~$\bf(J_1)$, $\bf(J_2)$, $\bf(J_3)$ and there exists a
  $\hat u \in \BV ([0,T]; \mathcal{U})$ such that $\mathcal{J} (\hat
  u) < +\infty$. Then, there exists a control minimizing
  $\mathcal{J}_n$:
  \begin{displaymath}
    \exists u^*_n \in \BV ([0,T];\mathcal{U}) \; \colon \quad
    \mathcal{J}_n(u_n^*) = \inf_{u\in \BV ([0,T];\mathcal{U})} \mathcal{J}_n(u) \,.
  \end{displaymath}
\end{theorem}

The above theorems yield the following corollary, which is the main
result of the present work. It ensures that the Escalator Boxcar Train
algorithm can also be used to solve optimal control problems.

\begin{corollary}
  \label{cor:Final}
  With the same assumptions and notation as in
  Theorem~\ref{thm:optPDE} and in Theorem~\ref{thm:optODE}, if
  \begin{displaymath}
    \lim_{n \to +\infty} {\Delta t}_n =0
    \qquad \mbox{ and } \qquad
    \lim_{n \to +\infty}
    d\left(
      \mu_o,
      \sum_{i=-n}^n m_o^i(t) \delta_{x_o^i(t)}
    \right)
    =
    0
  \end{displaymath}
  then,
  \begin{equation}
    \label{eq:limJ}
    \lim_{n\to +\infty} \mathcal{J}_n (u^*_n)
    =
    \inf_{u\in \BV ([0,T];\mathcal{U})} \mathcal{J}(u)
  \end{equation}
  and, up to a subsequence,
  \begin{equation}
    \label{eq:infJ}
    \lim_{n \to +\infty} \norma{u^*_n - u^*}_{\L\infty ([0,T]; \reali)} = 0
    \quad \mbox{ where } \quad
    \mathcal{J} (u^*)
    =
    \inf_{u\in \BV ([0,T];\mathcal{U})} \mathcal{J}(u) \,.
  \end{equation}
\end{corollary}

\section{The McKendrick -- Von F\"orster Model in Welfare Policies}
\label{sec:E}

The McKendric -- Von F\"orster model for population growth, equipped
with an integral functional to be maximized, provides a first example
of a system fitting within~\eqref{eq:5}, where the results in the
sections~\ref{sub:pde} and~\ref{sub:ode} can be applied.

Consider a population described by the amount $n = n (t,a)$ of people
that at time $t$ have the age $a$. Call $-d$, with $d = d (a)$, the
population mortality rate. We thus obtain:
\begin{displaymath}
  \left\{
    \begin{array}{l}
      \partial_t n + \partial_a n = -d (a) \, n
      \\
      \displaystyle
      n (t,0) = \int_0^{+\infty} \tilde\beta (a) \; n (t,a) \d{a}
      \\
      n (0,x) = n_o (x) \,.
    \end{array}
  \right.
\end{displaymath}
Here, $\tilde\beta$ describes the natality rate of the population of
age $a$ at time $t$.

Introduce a policy to sustain birth rate. It is then natural to assume
that a control parameter, say $u$, enters the birth functions. The
parameter $u$, possibly vector valued, reflects a government policy to
foster natality, helping through \emph{ad hoc} acts the families with
children.
\begin{equation}
  \label{eq:1}
  \left\{
    \begin{array}{l}
      \partial_t n_u + \partial_a n_u = -d (a) \, n_u
      \\
      \displaystyle
      n_u (t,0) = \int_0^{+\infty} \tilde\beta (a,u) \; n_u (t,a) \d{a}
      \\
      n_u (0,x) = n_o (x) \,.
    \end{array}
  \right.
\end{equation}
From the governmental point of view, the income of the state welfare
can be described by the functional
\begin{equation}
  \label{eq:J}
  \mathcal{J} (u)
  =
  \int_0^{+\infty} e^{-\lambda \, t}
  \left(
    \int_0^{+\infty} w (a) \; n_u (t,a) \d{a}
    - u (t) \, n_u (t,0)
  \right)  \d{t} \,.
\end{equation}
The weight $w = w (a)$ is positive all through the active age
interval, i.e., all during the period where individuals, paying taxes,
sustain the state. On the contrary, $w$ is negative when individuals
receive services from the state, e.g., during childhood and
retirement.

\begin{lemma}
  \label{lem:ex1}
  Fix a compact $\mathcal{U}$ in $\reali^N$ and $\alpha \in \left]0, 1
  \right[$. System~\eqref{eq:1} fits
  into~\eqref{eq:ode0}-\eqref{eq:odek} setting
  \begin{displaymath}
    b (t,\mu,u) (a)= 1
    \,,\qquad
    c (t,\mu,u) (a)= d (a)
    \,,\qquad
    \beta (t,\mu,u) (a) = \tilde \beta (a,u) \,.
  \end{displaymath}
  Moreover, if
  \begin{displaymath}
    d \in \C{0,1} (\reali_+, \reali)
    \,,\qquad \tilde\beta \in \C{0,1} (\mathcal{U} \times \reali_+; \reali)
  \end{displaymath}
  then, for all $u \in \Cb{1,\alpha} ([0,T];\mathcal{U})$,
  Theorem~\ref{thm:WP} applies.
\end{lemma}

In the present case, equations~\eqref{eq:ode0}--\eqref{eq:odek} take
the form, for $t \in [0, \Delta t]$
\begin{displaymath}
  \left\{
    \begin{array}{rcl@{\qquad}rcl}
      \dot x^i & = & 1
      &
      i & = & 0, \ldots, n
      \\
      \dot m^0 & = &
      \displaystyle
      -d (x^0) \, m^0 + \sum_{i=0}^n \tilde \beta \left(x^i,u (t)\right) m^i
      \\
      \dot m^i & = & -d (x^i) \, m^i
      &
      i & = & 1, \ldots, n
      \\
      x^i (0) & = & x^i_o & i & = &0, \ldots, n
      \\
      m^i (0) & = & m^i_o & i & = &0, \ldots, n
    \end{array}
  \right.
\end{displaymath}
while for $t \in [k\, \Delta t, (k+1) \Delta t]$ the solution to the
above system is extended as follows
\begin{displaymath}
  \left\{
    \begin{array}{rcl@{\qquad}rcl}
      \dot x^i & = & 1
      &
      i & = &-k, \ldots, n
      \\
      \dot m^{-k} & = &
      \displaystyle
      -d (x^{-k}) \, m^{-k}
      +
      \sum_{i=-k}^n \tilde \beta \left(x^i,u (t)\right) m^i
      \\
      \dot m^i & = & -d (x^i) \, m^i
      &
      i & = & -k+1, \ldots, n
      \\
      x^{-k} (k \, \Delta t) & = &0
      \\
      m^{-k} (k \, \Delta t) & = &0 \,.
    \end{array}
  \right.
\end{displaymath}
Note that the variables $x^i$ decouple and it is immediate to obtain
\begin{displaymath}
  x^i (t) = t - i\, \Delta t
  \qquad \mbox{ for } \quad t \geq \min \{-i\, \Delta t, 0\}
  \quad \mbox{ and } \quad  i = -k, \ldots,n\,.
\end{displaymath}
The discretized version of the cost functional~\eqref{eq:J} is
\begin{displaymath}
  \mathcal{J}^n (u)
  =
  \sum_{k =0}^{+\infty}
  \int_{k\, \Delta t}^{(k+1)\Delta t}
  e^{-\lambda\, t}
  \left(
    \sum_{i=-k+1}^n
    w(t - i\, \Delta t) \, m_u^i (t) - u (t) \, m_u^{-k} (t)
  \right)
  \d{t}
\end{displaymath}

\section{Technical Details}
\label{sec:T}

\subsection{Proofs Related to Section~\ref{sub:pde}}

\begin{lemma}
  \label{lem:BV}
  Fix $T>0$ and a normed space $X$. Let $x \in \BV([0,T];X)$. Then,
  for any $\epsilon>0$, there exists $n \in \naturali$, $\{t_1, t_2,
  \ldots, t_n\} \subset [0,T]$ and $\{x_1, x_2, \ldots, x_n\} \in X$
  such that, setting $x_\epsilon (t) = \sum_{i=1}^n x_i \,
  \chi_{\strut [t_{i-1}, t_i[} (t)$,
  \begin{displaymath}
    \sup_{t\in [0,T]}
    \norma{x (t) - x_\epsilon (t)}_{X}
    \leq \epsilon
    \,,\quad
    x_\epsilon ([0,T]) \subseteq x ([0,T])
    \quad \mbox{ and } \quad
    \tv_{X} (x_\epsilon) \leq \tv_{X} (x) \,.
  \end{displaymath}
\end{lemma}

\begin{proof}
  The construction of the function $x_\epsilon$ follows, for instance,
  from~\cite[Theorem~1.2, Chapter~1]{AmannEscher2}. The inclusion and
  the bound on the total variation are immediate, since the $x_i$ are
  chosen among the values attained by $x$.
\end{proof}

\begin{proofof}{Theorem~\ref{thm:pdeWP}}
  On the space $X = \C0(\mathcal{M}^+ (\reali_+) \times
  \reali_+;\reali)$, define the norm $\norma{\,\cdot\,}_X$ as
  in~\eqref{eq:dC0} and apply Lemma~\ref{lem:BV} to the maps
  $b,c,\beta \colon [0,T] \to X$. For every $\epsilon>0$, there exists
  $n \in \naturali$, $\{t_1, t_2, \ldots, t_n\} \subset [0,T]$ and
  piecewise constant functions $b_\epsilon, c_\epsilon,\beta_\epsilon
  \colon [0,T] \to X$ such that
  \begin{displaymath}
    \begin{array}{rcl@{,\qquad\qquad}rcl}
      \sup_{t\in[0,T]} \norma{b (t) - b_\epsilon (t)}_X & \leq & \epsilon
      &
      \tv_{X} (b_\epsilon) & \leq & \tv_{X} (b) \,,
      \\
      \sup_{t\in[0,T]} \norma{c (t) - c_\epsilon (t)}_X & \leq & \epsilon
      &
      \tv_{X} (c_\epsilon) & \leq & \tv_{X} (c) \,,
      \\
      \sup_{t\in[0,T]} \norma{\beta (t) - \beta_\epsilon (t)}_X & \leq & \epsilon
      &
      \tv_{X} (\beta_\epsilon) & \leq & \tv_{X} (\beta) \,.
    \end{array}
  \end{displaymath}
  Moreover, the inclusion proved in Lemma~\ref{lem:BV} ensures that
  \begin{eqnarray*}
    \sup_{t\in [0,T]}
    \norma{b_\epsilon (t)}_{\C{0,1}(\mathcal{M}^+ (\reali_+) \times
      \reali_+;\reali)}
    & \leq &
    \mathcal{L} \,,
    \\
    \sup_{t\in [0,T]}
    \norma{c_\epsilon (t)}_{\C{0,1}(\mathcal{M}^+ (\reali_+) \times
      \reali_+;\reali)}
    & \leq &
    \mathcal{L} \,,
    \\
    \sup_{t\in [0,T]}
    \norma{\beta_\epsilon (t)}_{\C{0,1}(\mathcal{M}^+ (\reali_+) \times
      \reali_+;\reali)}
    & \leq &
    \mathcal{L} \,.
  \end{eqnarray*}
  By construction, the sequences $b_\epsilon, c_\epsilon$ and
  $\beta_\epsilon$ converge to $b,c$ and $\beta$ uniformly on
  $[0,T]$. Hence, they are all Cauchy sequences.

  Fix $\epsilon>0$. For all $i=1, \ldots, n$,
  \cite[Theorem~2.11]{CarrilloColomboGwiazdaUlikowska},
  or~\cite[Theorem~4.6]{GwiazdaLorenz},
  \cite[Theorem~1.3]{GwiazdaMarciniak}, can be recursively applied on
  the interval $[t_{i-1}, t_i]$ to the problem
  \begin{displaymath}
    \left\{
      \begin{array}{l}
        \partial_t \mu_i
        +
        \partial_a \left(b_\epsilon (t,\mu_i) \, \mu_i\right)
        +
        c_\epsilon (t,\mu_i) \, \mu_i
        =
        0
        \\
        \displaystyle
        \left(b_\epsilon (t,\mu_i)\right) (0) \, D_\lambda\mu_i (0+)
        =
        \int_0^{+\infty} \beta_\epsilon (t,\mu_i) \, \d{\mu_i} (c)
        \\
        \mu_i (t_{i-1}) =\mu^o_{i-1}
      \end{array}
    \right.
  \end{displaymath}
  where $\mu^o_0 = \mu_o$ and $\mu^o_i = \lim_{t\to t_i-}\mu_{i-1}
  (t)$ for $i=1, \ldots, n-1$. Define $\mu^\epsilon (t)$ by
  $\mu^\epsilon(t) = \mu_i (t)$ whenever $t \in \left[t_{i-1},
    t\right[$.

  By~\cite[(iv) in Theorem~2.8]{CarrilloColomboGwiazdaUlikowska}, for
  any $\epsilon, \epsilon' >0$ sufficiently small,
  \begin{eqnarray*}
    d\left(\mu^\epsilon (t) , \mu^{\epsilon'} (t)\right)
    & \leq &
    C \, t \, e^{C\, t}
    \bigl( \,
    \sup_{t \in [0,T] }
    \norma{b_\epsilon (t) -b_{\epsilon'} (t)}_{\C0 (\mathcal{M}^+ (\reali_+)\times\reali_+; \reali)}
    \\
    & &
    \qquad\qquad
    +
    \sup_{t \in [0,T] }\norma{c_\epsilon (t) -c_{\epsilon'} (t)}_{\C0 (\mathcal{M}^+ (\reali_+)\times\reali_+; \reali)}
    \\
    & &
    \qquad\qquad
    +
    \sup_{t \in [0,T] }\norma{\beta_\epsilon (t) -\beta_{\epsilon'} (t)}_{\C0 (\mathcal{M}^+ (\reali_+)\times\reali_+; \reali)}
    \bigr) \,.
  \end{eqnarray*}
  Therefore, by the completeness of $\C0 \left([0,T]; \mathcal{M}^+
    (\reali_+)\right)$, there exists a measure valued map $\mu \in \C0
  \left([0,T]; \mathcal{M}^+ (\reali_+)\right)$ such that
  $\lim_{\epsilon \to 0} \sup_{t \in [0,T]} d\left(\mu^\epsilon (t) ,
    \mu (t)\right) =0$.

  To prove that $\mu$ solves~\eqref{eq:3} in the sense of
  Definition~\ref{def:pdeSol}, observe that by construction
  \begin{eqnarray*}
    & &
    \int_{\reali_+} \phi (t,a) \d{\mu^\epsilon_t (a)}
    -
    \int_{\reali_+} \phi (0,a) \d{\mu^\epsilon_o (a)}
    \\
    & = &
    \int_0^T \int_{\reali_+} \partial_t \phi (t,a) \d{\mu^\epsilon_t (a)} \d{t}
    \\
    & &
    +
    \int_0^T \int_{\reali_+}
    \left(
      \partial_a \phi (t,a) \; \left(b_\epsilon (t,\mu^\epsilon_t)\right) (a)
      +
      \phi (t,a) \; \left(c_\epsilon (t,\mu^\epsilon_t)\right) (a)
    \right)
    \d{\mu^\epsilon_t (a)} \d{t}
    \\
    & &
    +\int_0^T \int_{\reali_+}
    \phi (t,0) \; \left(\beta_\epsilon (t,\mu^\epsilon_t)\right) (a)
    \d{\mu^\epsilon_t (a)} \d{t} \,.
  \end{eqnarray*}
  and the limit $\epsilon\to0$ can pass inside the integral sign
  thanks to the uniform convergences $\mu^\epsilon \to \mu$,
  $b_\epsilon\to b$, $c_\epsilon\to c$ and $\beta_\epsilon \to \beta$
  on the time interval $[0,T]$.

  A further application of~\cite[(iv) in
  Theorem~2.1]{CarrilloColomboGwiazdaUlikowska} proves the stability
  estimate~\eqref{eq:pdeStab}.
\end{proofof}

\begin{proofof}{Corollary~\ref{cor:pdeWP}}
  Note first that if $b,c,\beta \in \mathcal{F}^u$ and $u \in \BV
  ([0,T];\mathcal{U})$, then the maps $b^u,c^u,\beta^u$ defined by
  $b^u (t) = b\left(t,u (t)\right)$, $c^u (t) = c\left(t,u (t)\right)$
  and $\beta^u (t) = \beta\left(t,u (t)\right)$ all satisfy
  $b^u,c^u,\beta^u \in \mathcal{F}$. Therefore,
  Theorem~\ref{thm:pdeWP} applies, ensuring the existence of a
  solution to~\eqref{eq:5}.

  Concerning the stability estimates, with obvious notations,
  by~\eqref{eq:pdeStab} we have:
  \begin{eqnarray*}
    d\left(\mu^1 (t) , \mu^{2} (t)\right)
    & \leq &
    d (\mu_o^1,\mu_o^2) \, e^{C\, t}
    +
    C \, t \, e^{C\, t}
    \bigl( \,
    \sup_{t \in [0,T]}
    \norma{b_1^{u_1} (t) - b_{2}^{u_2} (t)}_{\C0 (\mathcal{M}^+ (\reali_+)\times\reali_+;\reali)}
    \\
    & &
    \qquad\qquad\qquad\qquad\qquad
    +
    \sup_{t \in [0,T]}
    \norma{c_1^{u_1} (t) -c_{2}^{u_2} (t)}_{\C0 (\mathcal{M}^+ (\reali_+)\times\reali_+;\reali)}
    \\
    & &
    \qquad\qquad\qquad\qquad\qquad
    +
    \sup_{t \in [0,T]}
    \norma{\beta_1^{u_1} (t) -\beta_{2}^{u_2} (t)}_{\C0 (\mathcal{M}^+ (\reali_+)\times\reali_+;\reali)}
    \bigr) .
  \end{eqnarray*}
  Observe now that
  \begin{eqnarray*}
    & &
    \norma{b_1^{u_1} (t) - b_{2}^{u_2} (t)}_{\C0 (\mathcal{M}^+ (\reali_+)\times\reali_+;\reali)}
    \\
    & \leq &
    \norma{b_1^{u_1} (t) - b_1^{u_2} (t)}_{\C{0} (\mathcal{M}^+ (\reali_+) \times \reali_+;\reali)}
    +
    \norma{b_1^{u_2} (t) - b_2^{u_2} (t)}_{\C{0} (\mathcal{M}^+ (\reali_+) \times \reali_+;\reali)}
    \\
    & \leq &
    \mathcal{L}_u \, \norma{u_1 (t) - u_2 (t)}
    +
    \norma{b_1 (t) - b_2 (t)}_{\C0 (\mathcal{U} \times\mathcal{M}^+ (\reali_+)\times\reali_+;\reali)} \,.
  \end{eqnarray*}
  Entirely analogous estimates hold for the term $\norma{c_1^{u_1} (t)
    - c_{2}^{u_2} (t)}_{\C0 (\mathcal{M}^+
    (\reali_+)\times\reali_+;\reali)}$ as well as for
  $\norma{\beta_1^{u_1} (t) - \beta_{2}^{u_2} (t)}_{\C0 (\mathcal{M}^+
    (\reali_+)\times\reali_+;\reali)}$, allowing to
  obtain~\eqref{eq:pdeStabControl}.
\end{proofof}

\subsection{Proofs Related to Section~\ref{sub:ode}}

Aiming at the well posedness of~\eqref{eq:ode0}--\eqref{eq:odek} we
rewrite it as
\begin{equation}
  \label{eq:nice}
  \left\{
    \begin{array}{l}
      \dot x = f (t, x, m, u)
      \\
      \dot m = g (t, x, m, u)
      \\
      x (0) = x_o
      \\
      m (0) = m_o
    \end{array}
  \right.
\end{equation}
where
\begin{displaymath}
  \begin{array}{@{}r@{\;}c@{\;}l@{\quad}r@{\;}c@{\;}l@{}}
    x & = & (x^{-n}, \ldots, x^n)
    &
    m & = & (m^{-n}, \ldots, m^n)
    \\
    x_o^i & = &
    \left\{
      \begin{array}{@{\,}lr@{\,}c@{\,}l@{}}
        i \, \Delta t & i & = & 0, \ldots, n
        \\
        0 & i & = &-n, \ldots, -1
      \end{array}
    \right.
    &
    m_o^i & = &
    \left\{
      \begin{array}{@{\,}lr@{\,}c@{\,}l@{}}
        \mu_o \left(\left[i\, \Delta t, (i+1)\Delta t\right[\right)
        & i & = & 0, \ldots, n
        \\
        0 & i & = &-n, \ldots, -1
      \end{array}
    \right.
    \\
    f & \colon &
    [0,T] \times \reali_+^{2n+1} \times \reali_+^{2n+1} \times \mathcal{U}
    \to \reali^{2n+1}
    &
    g & \colon &
    [0,T] \times \reali_+^{2n+1} \times \reali_+^{2n+1} \times \mathcal{U}
    \to \reali^{2n+1}
  \end{array}
\end{displaymath}
the functions $f_i, g_i$ being defined, for $i=-n , \ldots, n$, by
\begin{equation}
  \label{eq:fi}
  f_i (t,x,m;u)
  =
  \left\{
    \begin{array}{l@{\qquad}r@{\,}c@{\,}l}
      \displaystyle
      \left(b \left(t, \sum_{j=-n}^n m^j \delta_{x^j}; u\right)\right) (x^i)
      & t & \geq &\max\{-i\, \Delta t,\, 0\}
      \\
      0
      & t & < &\max\{-i\, \Delta t,\, 0\}
    \end{array}
  \right.
\end{equation}
and
\begin{equation}
  \label{eq:gi}
  \!\!\!\!\!\!\!\!\!\!\!\!
  g_i (t,x,m;u)
  =
  \left\{
    \begin{array}{@{}l@{\;}r@{\,}c@{\,}l@{}}
      \displaystyle
      -c \left(t, \sum_{j=-n}^n m^j \delta_{x^j}; u\right) (x^i) \; m^i
      & t & > &\max\{(1-i)\, \Delta t,\, 0\}
      \\
      \begin{array}{@{}l@{}}
        \displaystyle
        -c \left(t, \sum\limits_{j=-n}^n m^j \delta_{x^j}; u\right) (x^i) \; m^i
        \\
        \displaystyle
        \quad+
        \sum_{\ell=-n}^n
        \beta \!
        \left(t, \sum_{j=-n}^n m^j \delta_{x^j}; u\right) \! (x^\ell) \; m^\ell
      \end{array}
      & t & \in &[\max\{-i \Delta t, 0\},\max\{(1-i)\Delta t, 0\}]
      \\
      0     & t & < &\max\{-i\, \Delta t,\, 0\}
    \end{array}
  \right.
  \!\!\!\!\!\!\!\!\!\!\!\!
\end{equation}

\begin{lemma}
  \label{lem:reg1}
  Fix positive $T, L$ and let $b,c,\beta \in \tilde{\mathcal{F}}^u$.
  Then, the map $f$ and $g$ defined in~\eqref{eq:fi} and~\eqref{eq:gi}
  satisfy the following conditions:
  \begin{enumerate}[$\bf(f_1)$]
  \item $t \to (f,g) (t,x,m;u)$ is measurable for all $x \in
    \reali_+$, $m \in \reali_+$ and $u \in \mathcal{U}$;
  \item $(x,m;u) \to (f,g) (t,x,m;u)$ is in $\C1$ for a.e.~$t \in
    [0,T]$;
  \item $(x,m;u) \to (f,g) (t,x,m;u)$ is sublinear, uniformly in $t$.
  \end{enumerate}
\end{lemma}

\begin{proof}
  We detail the proof that $f$ satisfies the above properties, the
  case of $g$ being entirely similar.

  The measurability of $t \to f (t,x,y;u)$ is immediate. To verify the
  differentiability, introduce the standard base $(e_{-n}, e_{-n+1}
  \ldots, e_{n-1}, e_n)$ of $\reali^{2n+1}$ and compute for $i = -n,
  \ldots, n$, for $t > \max\{-i\Delta t;0\}$ and for a (small) $h \in
  \reali$
  \begin{eqnarray*}
    & &
    \!\!\!
    f_i (t,x+h e_i,m;u) - f_i (t,x,m;u)
    \\
    & = &
    \!\!\!
    \left(b \! \left(t, \sum_{j=-n}^n m^j \delta_{x^j+\delta_{ij} h e_j}; u\right)\right) (x^i+he^i)
    -
    \left(b \! \left(t, \sum_{j=-n}^n m^j \delta_{x^j}; u\right)\right) (x^i)
    \\
    & = &
    \!\!\!
    \tilde b \! \left(t, \sum_{j=-n}^n m^j \, \bar b (x^j + \delta_{ij}h e_i), x^i+h e_i;u\right)
    -
    \tilde b \! \left(t, \sum_{j=-n}^n m^j \, \bar b  (x^j), x^i;u\right)
    \\
    & = &
    \!\!\!
    \partial_y
    \tilde b \! \left(t, \sum_{j=-n}^n m^j \, \bar b  (x^j), x^i;u\right)
    m^i \bar{b}' (x^i) h
    +
    \partial_a
    \tilde b  \left(t, \sum_{j=-n}^n m^j \, \bar b  (x^j), x^i;u\right) h
    +
    o (h)
  \end{eqnarray*}
  as $h \to 0$, while for $\ell \neq i$ and for $t > \max\{-i\Delta
  t;-\ell\Delta t, 0\}$
  \begin{eqnarray*}
    & &
    f_i (t,x+h e_\ell,m;u) - f_i (t,x,m;u)
    \\
    & = &
    \left(
      b \left(t, \sum_{j=-n}^n m^j \delta_{x^j+\delta_{\ell j} h e_\ell}; u\right)
    \right) (x^i)
    -
    \left(b \left(t, \sum_{j=-n}^n m^j \delta_{x^j}; u\right)\right) (x^i)
    \\
    & = &
    \tilde b\left(
      t, \sum_{j=-n}^n m^j \, \bar b (x^j + \delta_{\ell j}h e_\ell), x^i;u
    \right)
    -
    \tilde b\left(t, \sum_{j=-n}^n m^j \, \bar b  (x^j), x^i;u\right)
    \\
    & = &
    \partial_y
    \tilde b \left(t, \sum_{j=-n}^n m^j \, \bar b  (x^j), x^i;u\right)
    m^\ell \bar{b}' (x^\ell) h
    +
    o (h) \mbox{ as } h \to 0 \,,
  \end{eqnarray*}
  proving the differentiability of $f_i$ with respect to $x$.  Let now
  $i,\ell = -n , \ldots, n$:
  \begin{eqnarray*}
    & &
    f_i (t,x,m+h e_\ell;u) - f_i (t,x,m;u)
    \\
    & =&
    \left(b \left(t, \sum_{j=-n}^n (m^j+\delta_{\ell j} h e_\ell) \delta_{x^j}; u\right)\right) (x^i)
    -
    \left(b \left(t, \sum_{j=-n}^n m^j \delta_{x^j}; u\right)\right) (x^i)
    \\
    & = &
    \tilde b\left(
      t, \sum_{j=-n}^n (m^j + \delta_{\ell j}h e_\ell) \, \bar b(x^j), x^i;u
    \right)
    -
    \tilde b\left( t, \sum_{j=-n}^n m^j \, \bar b(x^j), x^i;u\right)
    \\
    & = &
    \partial_y \tilde b \left( t, \sum_{j=-n}^n m^j \, \bar b(x^i), x^i;u\right)
    \bar b (x^\ell) h
    +
    o (h) \mbox{ as } h \to 0 \,,
  \end{eqnarray*}
  so that $f_i$ is differentiable also with respect to $m$. The
  differentiability with respect to $u$ is immediate.

  Finally, we prove that $(x,m;u) \to (f,g) (t,x,m;u)$ is sublinear:
  \begin{eqnarray*}
    & &
    \modulo{f_i (t, x, m; u)}
    \\
    & \leq &
    \modulo{f_i(t, 0; u)}
    +
    \modulo{f_i (t, x, m; u) - f_i(t, 0; u)}
    \\
    & = &
    \modulo{
      \left(b(t, 0; u) (0)\right)
    }
    +
    \modulo{
      \left(b\left(t,\sum_{j=-n}^n m^j \delta_{x^j}; u\right) (x^i)\right)
      -
      \left(b(t, 0; u) (0)\right)
    }
    \\
    & = &
    \modulo{
      \tilde b (t, 0; u)
    }
    +\modulo{
      \tilde b \left(t, \sum_{j=-n}^n m^j \bar b (x^j), x^i; u\right)
      -
      \tilde b (t, 0; u)
    }
    \\
    & \leq &
    \modulo{\tilde b\left(t, 0, u\right)}
    +
    L
    \left(
      \modulo{\sum_{j=-n}^n m^j \bar b (x^j)}
      +
      x^i
    \right)
    \\
    & \leq &
    \modulo{\tilde b\left(t, 0, 0, u\right)}
    +
    L \sqrt{n}
    \left(
      \norma{m}
      +
      \norma{x}
    \right)
  \end{eqnarray*}
  and the first summand above is bounded by~$(\tilde F_1^u)$,
  completing the proof.
\end{proof}

Below, we call \emph{semiflow} (or process) on the set $M$ a map $S
\colon M \times [0,\delta]\times [0,T] \to M$ such that $S (0,t) =
\Id$ for all $t\in [0,T]$, and $S (t_3, t_1+t_2) \circ S (t_2, t_1) =
S (t_2+t_3, t_1)$ for all $t_1,t_2,t_3$ such that $t_1, t_1+t_2 \in
[0,T]$, $t_2, t_3, t_2+t_3 \in [0,\delta]$. We say that the semiflow
is Lipschitz continuous if the map $\mu \to S (t, t_o)\mu$ is
Lipschitz continuous, uniformly in $t_o \in [0,T]$ and in $t \in [0,
\delta]$.

\begin{lemma}
  \label{lem:semiflow}
  Let $(M, d_M) $ be a metric space and $S \colon M \times
  [0,\delta]\times [0,T] \to M$ a Lipschitz semiflow with Lipschitz
  constant $L$. For every Lipschitz continous map $\mu \colon [0,T]
  \to M$, the following estimate holds:
  \begin{equation}
    \label{eq:liminf}
    d_M \left(\mu_t, S(t,0)\mu_0 \right)
    \leq
    L
    \int_0^t \liminf_{h \to 0+}
    \frac{d_M \left(\mu_{\tau+h}, S(h,\tau)\mu_{\tau} \right)}{h} \; \d\tau
  \end{equation}
\end{lemma}

\noindent For a proof, see~\cite[Theorem~2.9]{BressanLectureNotes} or,
in the present non autonomous case,
\cite[Proposition~4.1]{GwiazdaEtAl} or~\cite[Proof of
Theorem~3.15]{ColomboGuerra2007}.

\begin{lemma}{\cite[Lemma~7.3]{GwiazdaEtAl}}
  \label{lem:ddeltas}
  Let $n \in \naturali$, $m, m'\in \reali^n$ and $x,x' \in
  \reali^n$. Then, with reference to the distance $d$ defined
  in~\eqref{eq:d},
  \begin{displaymath}
    d\left(
      \sum_{i=1}^n m_i \delta_{x_i},  \sum_{i=1}^n m'_i \delta_{x'_i}
    \right)
    \leq
    \max\left\{1, \sum_{i=1}^n \modulo{m_i}\right\}
    \;
    \sum_{i=1}^n
    \left(
      \modulo{m_i - m'_i} + \modulo{x_i - x'_i}
    \right) \,.
  \end{displaymath}
\end{lemma}

\begin{proofof}{Theorem~\ref{thm:Conv}}
  The proof relies on Lemma~\ref{lem:ddeltas}. First, we prove that
  the map
  \begin{equation}
    \label{eq:mun}
    \begin{array}{ccccc}
      \mu^n & \colon & [0,T] & \to &\mathcal M^+(\reali ^+)
      \\
      & & t & \to & \displaystyle \sum_{i=-n}^n m^i(t) \; \delta_{x^i(t)} \,,
    \end{array}
  \end{equation}
  where $t \to (x^i, m^i) (t)$
  solves~\eqref{eq:nice}--\eqref{eq:fi}--\eqref{eq:gi}, is Lipschitz
  continuous with respect to the metric $d$ defined
  in~\eqref{eq:d}. Indeed, by Lemma~\ref{lem:ddeltas}
  \begin{eqnarray}
    \nonumber
    d\left(\mu^n(t),\mu^n(s)  \right)
    & \leq&
    \max\left\{1, \sum_{i=-n}^n \modulo{m^i(t)}\right\}
    \;
    \sum_{i=-n}^n
    \left(
      \modulo{m^i(t) - m^i(s)} + \modulo{x^i(t) - x^i(s)}
    \right)
    \\
    \label{eq:thisOne}
    & \leq &
    \max\left\{
      1,
      \sum_{i=-n}^n \modulo{m^i(t)}
      \;
    \right\}
    \sum_{i=-n}^n \left( \Lip(x_i) + \Lip(m_i) \right)
    \;
    (t-s) \,.
  \end{eqnarray}
  Moreover,
  \begin{displaymath}
    \begin{array}{rcll}
      \Lip (x^i)
      & \leq &
      \norma{f_i}_{\C0([0,T]\times \reali_+ \times \reali_+ \times \mathcal{U}; \reali)}
      & \mbox{[by~\eqref{eq:nice}]}
      \\
      & \leq &
      \sup_{(t,u) \in [0,T]\times\mathcal{U}}
      \norma{b (t,u)}_{\C{0,1} (\mathcal{M}^+ (\reali_+)\times \reali_+; \reali)}
      & \mbox{[by~\eqref{eq:fi}]}
      \\
      & \leq &
      \mathcal{L}
      & \mbox{[by~$(\mathcal{F}^u_1)$]}
      \\[3pt]
      \modulo{m^i(t)}
      & \leq &
      \modulo{m_o} \exp\left(T
        \sup_{(t,u) \in [0,T]\times\mathcal{U}}
        \norma{c(t,u)}_{\C{0,1} (\mathcal{M}^+ (\reali_+)\times \reali_+; \reali)}
      \right)
      \\
      & &
      \quad
      \times
      \exp\left(
        (2n+1)T
        \sup_{(t,u) \in [0,T]\times\mathcal{U}}
        \norma{\beta (t,u)}_{\C{0,1} (\mathcal{M}^+ (\reali_+)\times \reali_+; \reali)}
      \right)
      & \mbox{[by~\eqref{eq:gi}]}
      \\
      & \leq &
      \modulo{m_o} e^{\left(2 (n+1) \mathcal{L}\right)T}
      & \mbox{[by~$(\mathcal{F}^u_1)$]}
      \\[6pt]
      \Lip (m^i)
      & \leq &
      \norma{g_i}_{\C0([0,T]\times \reali_+ \times \reali_+ \times \mathcal{U}; \reali)}
      & \mbox{[by~\eqref{eq:nice}]}
      \\
      & \leq &
      \sup_{(t,u) \in [0,T]\times\mathcal{U}}
      \norma{c (t,u)}_{\C{0,1} (\mathcal{M}^+ (\reali_+)\times \reali_+; \reali)}
      \sup_{t \in [0,T]} \modulo{m^i(t)}
      \\
      & &
      \quad
      +
      (2n+1)
      \norma{\beta (t,u)}_{\C{0,1} (\mathcal{M}^+ (\reali_+)\times \reali_+; \reali)}
      \sup_{t \in [0,T]} \modulo{m^i(t)}
      & \mbox{[by~\eqref{eq:gi}]}
      \\
      & \leq &
      2 (n+1) \mathcal{L}
      \modulo{m_o} e^{\left(2 (n+1) \mathcal{L}\right)T}
      & \mbox{[by~$(\mathcal{F}^u_1)$]}
    \end{array}
  \end{displaymath}
  These estimates, inserted in~\eqref{eq:thisOne}, complete the proof
  of the Lipschitz continuity of $\mu^n$ with respect to the metric
  $d$ defined in~\eqref{eq:d}.

  By the above computations and Corollary~\ref{cor:pdeWP}, we can thus
  use Lemma~\ref{lem:ddeltas}, where $S$ is the semiflow generated
  by~\eqref{eq:5} and $\mu$ is replaced by $\mu^n$ as defined
  in~\eqref{eq:mun}, obtaining
  \begin{eqnarray}
    \nonumber
    \!\!\!\!\!\!\!\!\!
    d\!\left(
      \sum_{i=-n}^n m^i(t) \, \delta_{x^i(t)} , \mu_u(t)
    \right)
    \!\!\!
    & = &
    \!\!\!
    d\left( \mu^n (t), S (t,0) \mu_o\right)
    \\
    \nonumber
    & \leq &
    \!\!\!
    d\left(\mu^n (t) , S (t,0) \mu^n (0)\right)
    +
    d\left(S (t,0) \mu^n (0), S (t,0) \mu_o\right)
    \\
    \label{eq:sleep}
    & \leq &
    \!\!\!
    e^{C t}
    \!
    \left[
      \int_0^t \!
      \liminf_{h\to 0+}
      \frac{d\left(\mu^n_{\tau+h}, S (h,\tau)\mu^n_\tau\right)}{h}
      \d\tau
      +
      d\left(
        \mu^n (0), \mu_o
      \right)
    \right].
  \end{eqnarray}
  The rest of the proof is devoted to estimate the integrand in the
  latter term above.

  Without loss of generality, we may assume that $\tau \in \left[0,
    \Delta t\right[$ and that $h$ is so small that $[\tau, \tau + h]
  \subset \left[0, \Delta t\right[$. Define
  \begin{displaymath}
    \mu_\tau(t) := S(t,\tau) \, \mu^n(\tau) \,.
  \end{displaymath}
  Then, for $t \in \left[\tau, \Delta t\right[$, the map $t \to
  \mu_\tau(t)$ solves problem~\eqref{eq:5} with initial datum
  $\mu^n(\tau) = \sum_{i=-n}^n m^i(\tau) \, \delta_{x^i(\tau)} =
  \sum_{i=0}^n m^i(\tau) \, \delta_{x^i(\tau)}$ assigned at time
  $\tau$.

  \newcommand{\n}{M} \newcommand{\f}{\pi}

  As in~\cite[Proof of Theorem~4.3]{GwiazdaEtAl}, $\mu_\tau(t)$ can be
  represented as
  \begin{displaymath}
    \mu_\tau(t)
    =
    \sum_{i=0}^n \n^i(\tau+h)\,\delta_{y^i(\tau+h)}
    +
    \f (t, \cdot ) \d{x}
  \end{displaymath}
  for suitable maps $\n^0, \ldots, \n^n$, $y^0, \ldots, y^n$, the
  density $\f (t, \cdot)$ arising from the boundary and supported
  inside $[x_o^0, y^0 (t)]$. Denote the total mass of $\f(t,\cdot)$ by
  \begin{displaymath}
    \n^{\f}(t) =\int_{x_o^0}^{y^0(t)} \f(t,x)\d{x} \,.
  \end{displaymath}
  Using suitable test functions in Definition~\ref{def:pdeSol}, we
  obtain:
  \begin{eqnarray*}
    y^i(\tau+h)
    \!\!\!\!
    & = &
    \!\!\!\!
    x^i(\tau)+\int_{\tau}^{\tau+h}b(t,\mu_\tau(t-\tau),u)(y^i(t)\d{t}
    \\
    \n^i(\tau+h)
    \!\!\!\!
    & = &
    \!\!\!\!
    m^i(\tau)+\int_{\tau}^{\tau+h}c(t,\mu_\tau(t-\tau),u)(y^i(t)\n^i(t)\d{t}
    \\
    \n^{\f}(\tau+h)
    \!\!\!\!
    & = &
    \!\!\!\!
    \n^{\f}\!(\tau)+\int_{\tau}^{\tau+h}
    \Big(
    \int_{x_o^0}^{y^0(t)}-c(t,\mu_\tau(t-\tau),u)(x)\d{\mu_\tau(t-\tau)(x)}
    \\
    & &
    \qquad\qquad\qquad\qquad\qquad
    +
    \int_{x_o^0}^{+\infty}
    \beta(t,\mu_\tau(t-\tau),u)(x)\d{\mu_\tau(t-\tau)(x)}
    \Big)
    \d{t}
    \\
    & = &
    \!\!\!\!
    \n^{\f}(\tau)
    \\
    & + &
    \!\!\!\!
    \int_{\tau}^{\tau+h} \!\!\! \int_{x_o^0}^{y^0(t)}
    \left[
      \left(
        -c(t,\mu_\tau(t-\tau),u)(x)
        +
        \beta(t,\mu_\tau(t-\tau),u)(x)
      \right)
      \d{\mu_\tau(t-\tau)(x)}
    \right]\!
    \d{t}
    \\
    & + &
    \!\!\!\!
    \sum_{i=0}^n\,
    \int_{\tau}^{\tau+h} \beta(t,\mu_\tau(t-\tau),u) \left(y^i(t)\right) \n^i(t) \d{t}
    \\
    & \leq &
    \!\!\!\!
    \n^{\f}(\tau)
    +
    2 \mathcal{L} \!\!
    \int_{\tau}^{\tau+h} \!\!\! \int_{x_o^0}^{y^0(t)} \!\!\!
    \f (t,x)
    \d{x} \d{t}
    +
    \sum_{i=0}^n \int_{\tau}^{\tau+h} \!\!\!\!
    \beta(t,\mu_\tau(t-\tau),u) \! \left(y^i(t)\right) \! \n^i(t) \d{t}
    \\
    & = &
    \!\!\!\!
    \n^{\f}(\tau)
    +
    \mathcal{O}(h^2)
    +
    \sum_{i=0}^n\,
    \int_{\tau}^{\tau+h}
    \beta(t,\mu_\tau(t-\tau),u) \left(y^i(t)\right) \n^i(t) \d{t} \, ,
  \end{eqnarray*}
  where with $\mathcal{O}(h^k)$ we denote a quantity that can be
  bounded the product of $h^k$ with a constant dependent only on $T,
  \mathcal{L}$ and $\mathcal{C}$.

  Above, $(\mathcal{F}^u_1)$ ensures a bound on $c$ and $\beta$. We
  also used the uniform boundedness of $\f (t, \cdot)$ on $[0,T]$ and
  the estimate
  \begin{displaymath}
    \modulo{y^0(t) - x_o^0}
    \leq
    \sup_{t\in[0,T]}
    \sup_{u\in\mathcal{U}}
    \norma{b(t,u)}_{\C0 (\mathcal{M}^+(\reali_+)\times\reali_+;\reali)} h
    \leq
    \mathcal{L}\, h
  \end{displaymath}
  for $t\in [\tau, \tau+h)$.  For $t\in \left[0,\Delta t\right[$
  define the time dependent measure
  \begin{equation}
    \label{eq:xi}
    \xi(t) = \sum_{i=0}^n p^i(t) \, \delta_{y^i(t)}
    \quad \mbox{ where } \quad
    \left\{
      \begin{array}{rcl}
        p^o(t) & = & \n^o(t)+\n^{\f}(t) \,,
        \\
        p^i(t) & = & \n^i(t), \quad \mathrm{for}\; i=1, \ldots, n
      \end{array}
    \right.
  \end{equation}
  in other words, in the measure $\xi(t)$ the mass created due to the
  boundary condition, described by the density $\f(t,\cdot )$, is
  shifted to the closest Dirac delta. We note that:
  \begin{equation}
    \label{eq:triangle}
    d\left(\mu_\tau (t), \mu^n(t+\tau)\right)
    \leq
    d\left(\mu_\tau(t), \xi(t+\tau)\right)
    +
    d\left(\xi(t+\tau), \mu^n(t+\tau)\right) \,.
  \end{equation}
  Recalling that $t \to y^i(t)$ is Lipschitz continous with Lipschitz
  constant
  \begin{displaymath}
    \Lip (y^i)
    \leq
    \sup_{t\in[0,T]} \sup_{u\in\mathcal{U}}
    \norma{b(t,u)}_{\C0(\mathcal{M}^+(\reali_+)\times\reali_+;\reali)}
    \leq
    \mathcal{L}
  \end{displaymath}
  and that the total mass is uniformly bounded on $[0,T]$, the first
  term in the right hand side of~\eqref{eq:triangle} is estimated as
  follows:
  \begin{eqnarray*}
    d \! \left(\mu_\tau(h), \xi(\tau+h)\right)
    \!\!\!\!
    & = &
    \!\!\!
    d\left(\f(\tau+h,\cdot),\n^{\f}(\tau+h)\,\delta_{y^0(\tau+h)}\right)
    \\
    & \leq &
    \!\!\!
    \modulo{y^0(\tau+h)}\,\n^{\f}(\tau+h)
    \\
    & \leq &
    \!\!\!
    \mathcal{L} \, \Delta t \!
    \left[
      \sup_{t\in[0,T],u\in\mathcal{U}}
      \norma{\beta(t,u)}_{\C0 (\mathcal{M}^+(\reali_+)\times\reali_+;\reali)}
      \int_{\tau}^{\tau+h}\sum_{i=0}^n \n^i(t) \d{t}
      +
      \mathcal{O}(h^2)
    \right]
    \\
    & \leq &
    \!\!\!
    \mathcal{L} \, \Delta t
    \left(\mathcal{L} \, C(T) \, h + \mathcal{O}(h^2) \right)
    \\
    & = &
    \!\!\!
    \Delta t \left(\mathcal{O}(h)+\mathcal{O}(h^2) \right) \, .
  \end{eqnarray*}
  To bound the second term in~\eqref{eq:triangle}, we want to use
  Lemma~\ref{lem:ddeltas}. Hence, we preliminary obtain the following
  estimates on $\modulo{x^i(\tau+h)-y^i(\tau+h)}$ and
  $\modulo{m^i(\tau+h)-p^i(\tau+h)}$:
  \begin{eqnarray*}
    \modulo{x^i(\tau+h)-y^i(\tau+h)}
    & \leq &
    \int_{\tau}^{\tau+h}
    \modulo{
      b(t,\mu^n(t),u)\left(x^i(t)\right)
      -
      b(t,\mu_\tau(t-\tau),u) \left(y^i(t)\right)}\d{t}
    \\
    & \leq &
    \int_{\tau}^{\tau+h}
    \modulo{
      b(t,\mu^n(t),u)\left(x^i(t)\right)
      -
      b(t,\mu_\tau(t-\tau),u)\left(x^i(t)\right)
    } \d{t}
    \\
    & &
    +
    \int_{\tau}^{\tau+h}
    \modulo{
      b(t,\mu_\tau(t-\tau),u)\left(x^i(t)\right)
      -
      b(t,\mu_\tau(t-\tau),u)\left(y^i(t)\right)}
    \d{t}
    \\
    & \leq &
    \mathcal{L}
    \int_{\tau}^{\tau+h}
    d\left(\mu^n(t),\mu_\tau(t-\tau)\right) \d{t}
    +
    \mathcal{L}\int_{\tau}^{\tau+h}\modulo{x^i(t)-y^i(t)} \d{t}
    \\
    & \leq &
    \mathcal{L}
    \int_{\tau}^{\tau+h}
    \left(
      \Lip_{\tau}(\mu^n)\, h
      +
      d\left(\mu^n(\tau), \mu_\tau(0)\right)
      +
      \Lip_{\tau}\left(\mu_\tau(t)\right)\, h
    \right)
    \d{t}
    \\
    & &
    +
    \int_{\tau}^{\tau+h}
    \left(
      \Lip_{\tau}(x^i)\, h
      +
      \modulo{x^i(\tau)-y^i(\tau)} + \Lip_{\tau}(y^i)\, h
    \right)
    \d{t}
    \\
    & \leq &
    \mathcal{O}(h^2) \,,
  \end{eqnarray*}
  since $\mu^n(\tau) = \mu_\tau(0)$ and $x^i(\tau) = y^i(\tau)$.
  Analogous estimates can be used to bound the term
  $\sum_{i=0}^{n}\modulo{m^i(\tau+h)-p^i(\tau+h)}$, taking into
  account~\eqref{eq:xi} and the estimate for $\n^{\f}(t)$:
  \begin{eqnarray*}
    & &
    \sum_{i=0}^{n} \modulo{m^i(\tau+h)-p^i(\tau+h)}
    \\
    & \leq &
    \sum_{i=0}^{n} \int_{\tau}^{\tau+h} \modulo{c(t,\mu^n(t),u)\left(x^i(t)\right)m^i(t) - c(t,\mu_\tau(t-\tau),u)\left(y^i(t)\right)\n^i(t)}\d{t}
    \\
    & &
    +
    \sum_{i=0}^n \int_{\tau}^{\tau+h}
    \modulo{\beta(t,\mu^n(t),u)\left(x^i(t)\right)m^i(t) - \beta(t,\mu_\tau(t-\tau),u)\left(y^i(t)\right)\n^i(t)}\d{t}
    +
    \mathcal{O}(h^2)
    \\
    & \leq &
    \sum_{i=0}^{n} \int_{\tau}^{\tau+h}
    \left(
      \modulo{c(t,\mu^n(t),u)\left(x^i(t)\right)}
      +
      \modulo{\beta(t,\mu^n(t),u)\left(x^i(t)\right) }
    \right)
    \modulo{m^i(t)-\n^i(t)}\d{t}
    \\
    & &
    +
    \sum_{i=0}^{n}
    \int_{\tau}^{\tau+h}
    \n^i(t)
    \modulo{
      c\left(t,\mu^n(t)\right)(x^i(t),u)
      -
      c(t,\mu^n(t),u)\left(y^i(t)\right)}
    \d{t}
    \\
    & &
    +
    \sum_{i=0}^{n} \int_{\tau}^{\tau+h}\n^i(t)\modulo{c(t,\mu^n(t),u)\left(y^i(t)\right) - c(t,\mu_\tau(t-\tau),u)\left(y^i(t)\right)}\d{t}
    \\
    & &
    +
    \sum_{i=0}^{n} \int_{\tau}^{\tau+h}\n^i(t)\modulo{\beta(t,\mu^n(t),u)\left(x^i(t)\right) - \beta(t,\mu^n(t),u)\left(y^i(t)\right)}\d{t}
    \\
    & &
    +
    \sum_{i=0}^{n} \int_{\tau}^{\tau+h}\n^i(t)\modulo{\beta(t,\mu^n(t),u)\left(y^i(t)\right) - \beta(t,\mu_\tau(t-\tau),u)\left(y^i(t)\right)}\d{t}+\mathcal{O}(h^2)
    \\
    & \leq &
    \norma{(c,\beta)}_{\C0}\sum_{i=0}^{n} \int_{\tau}^{\tau+h}\modulo{m^i(t)-\n^i(t)}\d{t}+\norma{(c,\beta)}_{\C{0,1}}\sum_{i=0}^{n} \int_{\tau}^{\tau+h}\n^i(t)\,\modulo{x^i(t)-y^i(t)}\d{t}
    \\
    & &
    +
    \norma{(c,\beta)}_{\C{0,1}}
    \sum_{i=0}^{n} \int_{\tau}^{\tau+h}
    \n^i(t) \, d\left(\mu^n(t),\mu_\tau(t-\tau)\right) \d{t}
    +
    \mathcal{O}(h^2)
    \\
    & \leq &
    \norma{(c,\beta)}_{\C0} \sum_{i=0}^{n} \int_{\tau}^{\tau+h}
    \left(
      \Lip(m^i) h
      +
      \modulo{m^i(\tau)-\n^i(\tau)}+\Lip(\n^i) h
    \right) \d{t}
    \\
    & &
    +
    \norma{(c,\beta)}_{\C{0,1}} \sum_{i=0}^{n} \int_{\tau}^{\tau+h}
    \n^i(t)
    \left(
      \Lip(x^i)h
      +
      \modulo{x^i(\tau)-y^i(\tau)}
      +
      \Lip(y^i) h
    \right)
    \d{t}
    \\
    & &
    +
    \norma{(c,\beta)}_{\C{0,1}}\sum_{i=0}^{n} \int_{\tau}^{\tau+h}
    \n^i(t)
    \left(
      \Lip(\mu^n)h
      +
      d\left(\mu^n(\tau),\mu_\tau(0)\right)
      +
      \Lip(\mu_\tau)h
    \right)
    \d{t}
    +
    \mathcal{O}(h^2)
    \\
    & \leq &
    \norma{(c,\beta)}_{\C{0}} \, h^2 \, \sum_{i=0}^{n}
    \left(\Lip(m^i) + \Lip(\n^i) \right)
    \\
    & &
    +
    \norma{(c,\beta)}_{\C{0,1}}
    h^2
    \left(
      2\norma{b}_{\C0}+\Lip(\mu^n)+\Lip(\mu_\tau)
    \right)
    \sum_{i=0}^{n}\n^i(t)+\mathcal{O}(h^2)
    \\
    & = &
    \mathcal{O}(h^2) \,.
  \end{eqnarray*}
  Inserting the obtained estimates in the integrand
  in~\eqref{eq:sleep}, we get:
  \begin{displaymath}
    \liminf_{h \to 0^+ } \frac{1}{h} \, d\left(\mu_\tau(h), \mu^n(\tau+h)\right)
    \leq
    \liminf_{h \to 0^+ } \frac{1}{h}
    \left(
      \Delta t
      \left(
        \mathcal{O}(h)
        +
        \mathcal{O}(h^2)
      \right)
      +
      \mathcal{O}(h^2)
    \right)
    =
    \O \,\Delta t \,,
  \end{displaymath}
  completing the proof.
\end{proofof}

\subsection{Proofs Related to Section~\ref{sub:control}}

\begin{proposition}{(Helly).}
  \label{prop:helly}
  Fix $N \in \naturali$ with $N>0$. Let $\mathcal{U}$ be a compact
  subset of $\reali^N$ and $T$ be a positive scalar.  Consider a
  sequence of functions $u_n\in \BV([0,T],\mathcal{U})$ such that
  $\sup_{n\in\naturali} \tv_{\reali^N} (u_n) <+\infty$. Then, there
  exists a subsequence $u_{n_k}$ and a function $u \in \BV ([0,T];
  \mathcal{U})$ such that
  \begin{displaymath}
    \lim_{k \to +\infty} \norma{u_{n_k}-u} = 0
    \quad \mbox{ and } \quad
    \tv_{\reali^N} (u) \leq \sup_{n\in\naturali} \tv_{\reali^N}(u_n) \, .
  \end{displaymath}
\end{proposition}

\noindent For a proof, see for instance~\cite[Chapter~2,
Theorem~2.3]{BressanLectureNotes}.

\begin{lemma}
  \label{lem:lsemc}
  Fix $T>0$ and for all $u \in \BV([0,T]; \mathcal{U})$ call $\mu_u$
  the corresponding solution to~\eqref{eq:5}. Assume there exists a
  constant $\mathcal{L}$ such that for all $u_1, \, u_2 \in \BV([0,T];
  \mathcal{U})$
  \begin{displaymath}
    d(\mu_{u_1}, \mu_{u_2})
    \leq
    \mathcal{L} \; \norma{u_1 - u_2}_{\L\infty ([0,T]; \reali)}.
  \end{displaymath}
  Let $\gamma \in \C{0,1}(\reali_+; \reali_+)$ and $j$
  satisfy~$\bf(J_1)$--$\bf (J_3)$. Then, the functional $\mathcal{J}$
  defined in~\eqref{eq:8}--\eqref{eq:9} is lower semi-continuous with
  respect to the $\L\infty$--norm.
\end{lemma}

\begin{proof}
  Let $u_n \in \BV ([0,T]; \mathcal{U})$ be a sequence converging to
  $u \in \BV ([0,T]; \mathcal{U})$ in the $\L\infty$-norm.  First we
  recall that by~\cite[Theorem~1, \S~5.2.1]{EvansGariepy}
  $\tv_{\reali^N} (u) \leq \liminf_{n\to\infty} \tv_{\reali^N} (u_n)$.

  Next we show the sequential continuity of the map
  $\widetilde{\mathcal{J}}$ defined in~\eqref{eq:8}, using $\bf(J3)$
  and the fact that $\omega$ is a nondecreasing function by~$\bf
  (J_3)$.
  \begin{eqnarray*}
    \modulo{\widetilde{\mathcal{J}} (u_n) - \widetilde{\mathcal{J}} (u)}
    \!\!\!\!
    & \leq &\!\!\!
    \int_0^{T}
    \modulo{
      j \! \left(t,u_n(t),\int_{\reali_+}\gamma(\xi) \d{\mu_{u_n}(t)(\xi)}\right)-
      j \! \left(t,u(t),\int_{\reali_+}\gamma(\xi) \d{\mu_u(t)(\xi)}\right)
    } \d{t}
    \\
    & \leq &\!\!\!
    \int_0^{T}
    \!\!
    L(t)\,
    \omega \!
    \left(
      \modulo{
        \int_{\reali_+} \!\! \gamma(\xi)\d{\mu_{u_n}(t)(\xi)}
        -
        \int_{\reali_+} \!\! \gamma(\xi)\d{\mu_{u}(t)(\xi)}
      }
      +
      \modulo{u_n(t)-u(t)}
    \right)
    \d{t}
    \\
    & \leq &\!\!\!
    \omega \!
    \left( \!
      \norma{\gamma}_{\L\infty(\reali_+; \reali_+)}
      \sup_{t\in[0,T]}
      d\left(\mu_{u_n}(t),\mu_{u}(t)\right)
      +
      \norma{u_n - u}_{\L\infty ([0,T];\reali^N)}
      \!\right)
    \! \int_0^{T} \!\!\! L(t)\d{t}
    \\
    & \leq &\!\!\!
    \omega \!
    \left(
      (1 + \mathcal{L} \norma{\gamma}_{\L\infty(\reali_+; \reali_+)})
      \norma{u_n - u}_{\L\infty ([0,T];\reali^N)}    \right)
    \! \int_0^{T} \!\!\! L(t)\d{t}
    \\
    & \to &\!\!\!
    0
    \quad \mbox{ in } \L\infty ([0,T]; \reali^N)
    \mbox{ as } n \to +\infty\,,
  \end{eqnarray*}
  completing the proof.
\end{proof}

\begin{proofof}{Theorem~\ref{thm:optPDE}}
  Let $\epsilon_n$ be a strictly decreasing sequence converging to
  $0$. Correspondingly, there exists a sequence $u_{\epsilon_n} \in
  \BV ([0,T]; \mathcal{U})$ such that
  \begin{displaymath}
    \mathcal{J}(u_{\epsilon_{n}})\leq \inf_{u}\mathcal{J}(u)+\epsilon_n
  \end{displaymath}
  and, without loss of generality, we may also assume that
  $\mathcal{J} (u_{\epsilon_n}) \leq \mathcal{J} (\hat u) +1$ for all
  $n$.  Moreover, by~$\bf(J_1)$ and~\eqref{eq:9}
  \begin{displaymath}
    \tv(u_{\epsilon_n})
    \leq
    \mathcal{J}(u_{\epsilon_n})
    \leq
    \mathcal{J}(\hat u) +1
  \end{displaymath}
  So that Proposition~\ref{prop:helly} can be applied, showing that,
  up to a subsequence, $u_{\epsilon_n}$ converges pointwise a.e.~to a
  function $u^* \in \BV ([0,T]; \mathcal{U})$. Therefore,
  \begin{displaymath}
    \begin{array}{rcl@{\qquad\quad}l}
      \mathcal{J} (u^*)
      & = &
      \displaystyle
      \mathcal{J} (\lim_{n\to+\infty} u_{\epsilon_n})
      &
      [\mbox{by the definition of }u^*]
      \\[6pt]
      & \leq &
      \displaystyle
      \liminf_{n\to\infty} \mathcal{J}(u_{\epsilon_n})
      & [\mbox{by Lemma~\ref{lem:lsemc}}]
      \\[6pt]
      & = &
      \displaystyle
      \inf_{u \in \BV ([0,T];\mathcal{U})} \mathcal{J} (u)
      &
      [\mbox{by the definition of }u_{\epsilon_n}]\,,
    \end{array}
  \end{displaymath}
  completing the proof.
\end{proofof}

\begin{proofof}{Theorem~\ref{thm:optODE}}
  The proof follows the same lines as that of
  Theorem~\ref{thm:optPDE}.
\end{proofof}

\begin{proofof}{Corollary~\ref{cor:Final}}
  We first prove the uniform convergence $\mathcal{J}_n \to
  \mathcal{J}$ of the costs on $\BV ([0,T]; \mathcal{U})$,
  using~$\bf(J_3)$, \eqref{eq:8}, \eqref{eq:9}, \eqref{eq:10},
  \eqref{eq:11} and Theorem~\ref{thm:Conv}, for all $n$, we have:
  \begin{eqnarray*}
    & &
    \sup_{u \in \BV ([0,T];\mathcal{U})}
    \modulo{\mathcal{J}_n(u) - \mathcal{J}(u)}
    \\
    & = &
    \sup_{u \in \BV ([0,T];\mathcal{U})}
    \modulo{\widetilde{\mathcal{J}}_n(u) - \widetilde{\mathcal{J}}(u)}
    \\
    & \leq &
    \sup_{u \in \BV ([0,T];\mathcal{U})}
    \int_0^T
    \modulo{
      j\left(t,u(t),\int_{\reali_+}\gamma(\xi) \d{\mu^n_{u}(t)(\xi)}\right)
      -
      j\left(t,u(t),\int_{\reali_+}\gamma(\xi) \d{\mu_u(t)(\xi)}\right)
    } \d{t}
    \\
    & \leq &
    \sup_{u \in \BV ([0,T];\mathcal{U})}
    \int_0^{T} L(t) \;
    \omega \! \left(
      \modulo{
        \int_{\reali_+}\gamma(\xi)\d{\mu^n_{u}(t)(\xi)}
        -
        \int_{\reali_+}\gamma(\xi)\d{\mu_{u}(t)(\xi)}
      }\right)
    \d{t}
    \\
    & \leq &
    \int_0^{T} L(t) \d{t} \;
    \sup_{u \in \BV ([0,T];\mathcal{U})}
    \omega \! \left(
      \norma{\gamma}_{L^{\infty}} \,
      \sup_{t\in[0,T]} d\left(\mu^n_{u}(t),\mu_{u}(t)\right)
    \right)
    \\
    & \leq &
    \int_0^{T} L(t) \d{t} \;
    \omega \!\left(
      C \;\norma{\gamma}_{L^{\infty}} \,
      \left[
        {\Delta t}_n
        +
        d\left(\mu_o, \sum_{i=0}^n m_o^i \, \delta_{x_o^i}\right)
      \right]
    \right)
    \\
    & \to & 0
    \quad \mbox{ as } n \to +\infty \,,
  \end{eqnarray*}
  which immediately implies~\eqref{eq:limJ}.

  Using~$\bf(J_2)$, the same procedure used in the proof of
  Theorem~\ref{thm:optPDE} ensures that $\tv (u_n)$ is bounded
  uniformly in $n$. By Proposition~\ref{prop:helly}, up to a
  subsequence, $u_n \to \bar u$ a.e.~on $[0,T]$,
  proving~\eqref{eq:infJ}. Using Lemma~\ref{lem:lsemc} and the uniform
  convergence of $\mathcal{J}_n$ to $\mathcal{J}$ proved above, we
  have:
  \begin{eqnarray*}
    \mathcal{J}(\bar u)
    & \leq &
    \liminf_{n\to\infty} \mathcal{J}(u_n)
    \\
    & = & \liminf_{n\to\infty}
    \left(
      \mathcal{J}(u_n)
      +
      \mathcal{J}_n(u_n)
      -
      \mathcal{J}_n(u_n)
    \right)
    \\
    & \leq &
    \liminf_{n\to\infty}
    \mathcal{J}_n (u_n)
    +
    \lim_{n\to\infty}
    \left(
      \sup_{u \in \BV ([0,T];\mathcal{U})}
      \modulo{\mathcal{J}_n(u)-\mathcal{J}(u)}
    \right)
    \\
    & \leq &
    \lim_{n\to\infty} \mathcal{J}_n(u_n)
    \\
    & = &
    \inf_{u \in \BV ([0,T];\mathcal{U})}
    \mathcal{J}(u) \,,
  \end{eqnarray*}
  where~\eqref{eq:limJ} was used to obtain the last equality.
\end{proofof}

\appendix
\section{Appendix: ODE Results}

For completeness, we collect here a few basic ODE results using
exactly the spaces and norms of use above.

\begin{lemma}
  \label{lem:Calculus2}
  Fix $T>0$ and a compact $\mathcal{U}\subset \reali^M$. Let $f \colon
  [0,T] \times \reali^N \times \mathcal{U} \to \reali^N$ be such that
  \begin{enumerate}[$\bf(f_1)$]
  \item $t \to f (t,x;u)$ is measurable for all $x \in \reali^N$ and
    $u \in \mathcal{U}$;
  \item $(x,u) \to f (t,x;u)$ is in $\C1$ for a.e.~$t \in [0,T]$ and
    there exists a constant $L>0$ such that for a.e.~$t \in [0,T]$,
    for all $x_1,x_2 \in \reali^N$ and for all $u_1,u_2 \in
    \mathcal{U}$,
    \begin{displaymath}
      \norma{f (t,x_1;u_1) - f (t,x_2;u_2)}
      \leq
      L \left(\norma{x_1-x_2} + \norma{u_1-u_2}\right) \,.
    \end{displaymath}
  \end{enumerate}
  Then, for all $x_o \in \reali^N$ and all $u \in \L\infty
  ([0,T];\mathcal{U})$, the problem
  \begin{equation}
    \label{eq:calculusCP}
    \left\{
      \begin{array}{l}
        \dot x = f (t,x;u)
        \\x (0) = x_o
      \end{array}
    \right.
  \end{equation}
  admits a unique solution $X (u) \colon [0,T] \to \reali^N$. The map
  $X \colon \L\infty ([0,T];\mathcal{U}) \to \C1 ([0,T];\reali^N)$ is
  Gateaux differentiable in any direction $v \in \L\infty
  ([0,T];\mathcal{U})$ and the directional derivative $D_vX (u)$
  solves the Cauchy problem
  \begin{equation}
    \label{eq:calculusG}
    \left\{
      \begin{array}{l}
        \frac{\d{\,}}{\d{t}} D_vX (u) =
        \partial_x f\left(t, X (u); u\right) D_vX (u)
        +
        \partial_v f\left(t, X (u); u\right) v
        \\[6pt]
        \left(D_v X (u)\right) (0) = 0 \,.
      \end{array}
    \right.
  \end{equation}
\end{lemma}

\begin{proof}
  The map $X$ is well defined by the standard theory of Caratheodory
  ODEs, see for instance~\cite{Filippov}. Moreover, there exists a
  compact $\Omega \subset \reali^N$ such that for all $u \in \L\infty
  ([0,T];\mathcal{U})$, $\left(X (u)\right) ([0,T]) \subset \Omega$.

  To prove the directional differentiability, call $g$ the solution to
  the linear problem~\eqref{eq:calculusG} and use the integral form of
  the Cauchy problem~\eqref{eq:calculusCP} to obtain
  \begin{eqnarray*}
    & &
    \frac{1}{h} \left(X (u+hv) - X (u)\right) (t)
    -
    g (t)
    \\
    & = &
    \int_0^t
    \frac{f\left(\tau, \left(X (u+ h v)\right) (\tau), (u+hv) (\tau)\right)
      -
      f\left(\tau, \left(X (u)\right) (\tau), u (\tau)\right)}{h}
    \d\tau
    -
    g (t)
    \\
    & = &
    \int_0^t
    \int_0^1
    \partial_x f\left(
      \tau, \left(X (u+ \theta h v)\right) (\tau), (u+ \theta h v) (\tau)
    \right)
    \d\theta
    \frac{\left(X (u+hv)-X (u)\right) (\tau)}{h} \d{\tau}
    \\
    & &
    +
    \int_0^t \int_0^1
    \partial_v f\left(
      \tau, \left(X (u+ \theta h v)\right) (\tau), (u+ \theta h v) (\tau)
    \right)
    \d\theta
    v (\tau) \d{\tau} - g (t)
    \\
    & = &
    \int_0^t
    \int_0^1
    \left[
      \partial_x f\left(
        \tau, \left(X (u+ \theta h v)\right) (\tau), (u+ \theta h v) (\tau)
      \right)
      -
      \partial_x f\left(\tau, X (u); u\right)
    \right]
    \d\theta
    \\
    & &
    \qquad\qquad\qquad\qquad
    \times
    \frac{\left(X (u+hv)-X (u)\right) (\tau)}{h} \d{\tau}
    \\
    & &
    +
    \int_0^t
    \partial_x f\left(\tau, X (u); u\right)
    \left(
      \frac{\left(X (u+hv)-X (u)\right) (\tau)}{h} - g (\tau)
    \right)
    \d\tau
    \\
    & &
    +
    \int_0^t \int_0^1
    \left[
      \partial_v f\left(
        \tau, \left(X (u+ \theta h v)\right) (\tau), (u+ \theta h v) (\tau)
      \right)
      -
      \partial_v f\left(\tau, X (u); u\right)
    \right]
    \d\theta
    v (\tau) \d{\tau}
  \end{eqnarray*}
  Lusin's Theorem~\cite[Theorem~1]{LoebTalvila}, applied to
  $(\partial_x f, \partial_v f) \in \L\infty \left([0,T]; \C{0,1}
    (\reali^N\times\mathcal{U};\reali^{2N^2})\right)$, ensures that
  for any $\epsilon>0$, there exists a compact set $K \subset [0,T]$
  such that the Lebesgue measure of $[0,T] \setminus K $ is smaller
  than $\epsilon$ and both $\partial_x f$ and $\partial_v f$ are
  continuous on $K \times \Omega \times \mathcal{U}$, hence also
  uniformly continuous. Therefore, if $h$ is sufficiently small,
  \begin{eqnarray*}
    \sup_{K \times \Omega \times \mathcal{U}}
    \sup_{\theta \in [0,1]}
    \norma{
      \partial_x f\left(
        \tau, \left(X (u+ \theta h v)\right) (\tau), (u+ \theta h v) (\tau)
      \right)
      -
      \partial_x f\left(\tau, X (u); u\right)}
    & \leq & \epsilon \,,
    \\
    \sup_{K \times \Omega \times \mathcal{U}}
    \sup_{\theta \in [0,1]}
    \norma{
      \partial_v f\left(
        \tau, \left(X (u+ \theta h v)\right) (\tau), (u+ \theta h v) (\tau)
      \right)
      -
      \partial_v f\left(\tau, X (u); u\right)}
    & \leq & \epsilon \,.
  \end{eqnarray*}
  Introduce now the quantity
  \begin{displaymath}
    \delta_h (t)
    =
    \norma{    \frac{1}{h} \left(X (u+hv) - X (u)\right) (t)
      -
      g (t)
    } \,.
  \end{displaymath}
  Since $\norma{\partial_x f} \leq L$, $\norma{\partial_v f} \leq L$,
  $\norma{f}_{\L\infty ([0,T]\times\Omega\times\mathcal{U};\reali^N)}
  < +\infty$, the above estimates lead to
  \begin{eqnarray*}
    \delta_h (t)
    & \leq &
    (2 \, L \, \epsilon + \epsilon \, t)
    \norma{f}_{\L\infty ([0,T]\times\Omega\times\mathcal{U};\reali^N)}
    +
    \int_0^t L \, \delta_h (\tau) \d\tau
    +
    (2 \, L \, \epsilon + \epsilon \, t) \norma{v}_{\L\infty ([0,T];\reali^M)}
    \\
    & = &
    (2 \, L + t)
    \left(
      \norma{f}_{\L\infty ([0,T]\times\Omega\times\mathcal{U};\reali^N)}
      +
      \norma{v}_{\L\infty ([0,T];\reali^M)}
    \right)
    \epsilon
    +
    \int_0^t L \, \delta_h (\tau) \d\tau \,.
  \end{eqnarray*}
  An application of Gronwall Lemma yields that for all $\epsilon>0$,
  if $h$ is sufficiently small
  \begin{displaymath}
    \delta_h (t) \leq
    (2 \, L + t)
    \left(
      \norma{f}_{\L\infty ([0,T]\times\Omega\times\mathcal{U};\reali^N)}
      +
      \norma{v}_{\L\infty ([0,T];\reali^M)}
    \right)
    \epsilon \, e^{L\,t}
  \end{displaymath}
  completing the proof.
\end{proof}

\smallskip

\noindent\textbf{Acknowledgment:} R.M.C.~was partially supported by
the 2015~GNAMPA project \emph{Balance Laws in the Modeling of
  Physical, Biological and Industrial Processes} and by the CaRiPLo
project 2013-0893. The research of M.R.~received funding from the
National Science Centre, DEC-2012/05/E/ST1/02218. The research of
P.G.~received funding from the National Science Centre, Poland,
2014/13/B/ST1/03094.

{\small

  \bibliographystyle{abbrv}

  \bibliography{control4}

}

\end{document}